\documentclass{article}
\usepackage[utf8]{inputenc}
\usepackage{hyperref}
\hypersetup{
    colorlinks=true,
    linkcolor=blue,
    citecolor=blue,      
    urlcolor=blue,
}

\usepackage{amsfonts}
\usepackage{amsrefs}
\usepackage{mathrsfs,amsmath}
\usepackage{amssymb}
\usepackage{titlesec}
\usepackage{graphicx}
\usepackage{wrapfig}
\usepackage{geometry}
\usepackage{bm}
\usepackage{amsthm}
\usepackage{commath}
\usepackage{color}
\newtheorem{theorem}{Theorem}[section]
\newtheorem*{remark}{Remark}
\newtheorem{lemma}[theorem]{Lemma}
\newtheorem{cor}[theorem]{Corollary}
\newtheorem*{conjecture}{Conjecture}

\title{Decomposing Centrally Symmetric Convex Polyhedral Surfaces into Parallelograms}
\author{\Large Zili Wang \footnote{Zili Wang is supported by the Natural Science Foundation of Guangdong Province (Grant
No. 2023A1515010658)} \\School of Science\\ Sun Yat-sen University\\ wangzli6@mail.sysu.edu.cn \and \Large Cong Wu \\ Mathematics Department \\ Shenzhen Middle School \\ cwu@shenzhong.net}

\begin{document}
\maketitle

\begin{abstract}
  Let $\mathcal{M}_{2N}(\delta_1, \delta_2,\dots, \delta_N)$ be the moduli space of centrally symmetric convex polyhedral surfaces with $2N$ labeled vertices and prescribed cone-deficits $\delta_1$, $\delta_2$, $\dots$, $\delta_N$. We show that $\mathcal{M}_{2N}(\delta_1, \delta_2,\dots, \delta_N)$ has the structure of a real hyperbolic manifold of dimension $2N-3$. When $N=4$ and $5$, we show that every surface in $\mathcal{M}_{2N}(\delta_1, \delta_2,\dots, \delta_N)$ can be decomposed into at most $2\binom{2N-2}{2}$ parallelograms, and the decomposition is invariant under the antipodal map. Using the edge-lengths of these parallelograms as coordinates, we show that the moduli space of centrally symmetric polyhedral surfaces with $8$ unlabeled vertices and cone-deficits $\frac{\pi}{2}$ is isometric to the quotient of a real hyperbolic regular ideal $5$-simplex by the dihedral group $D_6$. 
\end{abstract}


\section{Introducton}

The surface of a polyhedron is a Euclidean cone sphere. It has a metric that is locally Euclidean except at every vertex $v_i$, where it is locally isometric to a Euclidean cone with angle $\theta_i$. The value $\delta_i=2\pi-\theta_i$ is called the \textbf{cone-deficit} of $v_i$. It follows from the Gauss–Bonnet Theorem that the sum of all cone-deficits equals $4\pi$.

A polyhedron is convex if $0<\delta_i<2\pi$ for all $i$. In \cite{Th1}, Thurston considered the moduli space of convex polyhedra with prescribed cone-deficits. By decomposing the surfaces of polyhedra into triangles, he built coordinates to describe these polyhedra, and showed that their moduli space admits a Hermitian metric, with respect to which it is a complex hyperbolic manifold. In addition, under certain conditions, the metric completion of this moduli space is an orbifold. This result provides a geometric way to interpret the monodromy groups of hypergeometric functions investigated by Picard \cite{Pic1}\cite{Pic2}, Mostow and Deligne \cite{DM1}\cite{DM2}\cite{Mos}. In addition, this result has been applied in counting and enumerating specific structures on the sphere, such as counting convex tilings \cite{ES2} and enumerating fullerenes \cite{ES1}. Thurston's idea in building coordinates and the Hermitian form is also used to study the moduli space of convex polygons \cite{BG}.

Inspired by this work, we study the space of centrally symmetric convex polyhedra. A polyhedron is \textbf{centrally symmetric} if it is isometric to a polyhedron symmetric about the origin, which means that a point $x$ lies on its surface if and only if its antipodal point $-x$ does. The map that sends $x$ to $-x$ for all $x$ on the surface is called the \textbf{antipodal map}. Thus, a centrally symmetric polyhedron has an even number $2N$ of vertices, and its cone-deficits are described by $N$ positive numbers that sum up to $2\pi$. 

The main results of this paper can be summarized as follows. In Section 2, we show that the moduli space of centrally symmetric convex polyhedra with prescribed cone-deficits has the structure of a real hyperbolic manifold. The way we compute the signature of the quadratic form is very similar to the way Thurston computed it for the Hermitian form in \cite{Th1}. In addition, it is known that every centrally symmetric convex polygon can be decomposed into finitely many parallelograms. This leads us to conjecture that a similar result holds for every centrally symmetric convex polyhedral surface. In a previous work \cite{Wang}, we showed that the surface of every centrally symmetric octahedron has a parallelogram decomposition that is invariant under the antipodal map. In Section 3, we establish a method that relates a parallelogram decomposition of a polyhedral surface to a structure called ``loop arrangement'' on the sphere. Using this method, we generalize this result to centrally symmetric convex polyhedra with $8$ and $10$ vertices in Section 4 and 5. In these cases, it turns out that the edge-lengths of the parallelograms in the decomposition can be used as coordinates to describe polyhedra. These coordinates are different from Thurston's coordinates in \cite{Th1}, but may sometimes provide a more concrete geometric description of the moduli spaces when the polyhedra are centrally symmetric. For example, in Section 4, we show that the moduli space of centrally symmetric polyhedra with $8$ vertices whose cone-deficits are $\frac{\pi}{2}$ is isometric to the quotient of a real hyperbolic regular ideal $5$-simplex by a group isomorphic to $D_6$. 

\section{Space of Centrally Symmetric Convex Polyhedra}

Let $\delta_1$, $\delta_2$, $\dots$, $\delta_N$ be $N$ positive numbers that sum up to $2\pi$. Let $\mathcal{C}_{2N}(\delta_1, \delta_2, \dots, \delta_N)$ be the space of centrally symmetric convex polyhedral surfaces with $N$ pairs of antipodal vertices labeled by $i^+$ and $i^-$ (where $1\leq i\leq N$) with prescribed cone-deficit $\delta_i$. In this space, two surfaces are considered equivalent if and only if they differ by a Euclidean isometry that preserves the labels of all the vertices. 

\begin{theorem} \label{thm1}
    The space $\mathcal{C}_{2N}(\delta_1, \delta_2, \dots, \delta_N)$ has the structure of a complex manifold of dimension $N-1$.
\end{theorem}
\begin{proof}
    Let $\mathscr{S}$ be a surface in $\mathcal{C}_{2N}(\delta_1, \delta_2, \dots, \delta_N)$. First, we will associate $\mathscr{S}$ with $N-1$ complex numbers.

    Consider the set of directed cycles on $\mathscr{S}$, each of which: 1) is invariant under the antipodal map, 2) consists of directed edges on $\mathscr{S}$ (line segments connecting the vertices of $\mathscr{S}$), and 3) visits all vertices of $\mathscr{S}$ except $i^+$ and $i^-$ for exactly one $i\in\{1,2,\dots,N\}$. Figure \ref{fig1} shows two cycles satisfying the conditions above, where $\mathscr{S}$ is the surface of a cube in $\mathcal{C}_8(\frac{\pi}{2}, \frac{\pi}{2}, \frac{\pi}{2}, \frac{\pi}{2})$.

    \begin{figure}[h]
    \centering    
    \includegraphics[width=0.6\textwidth]{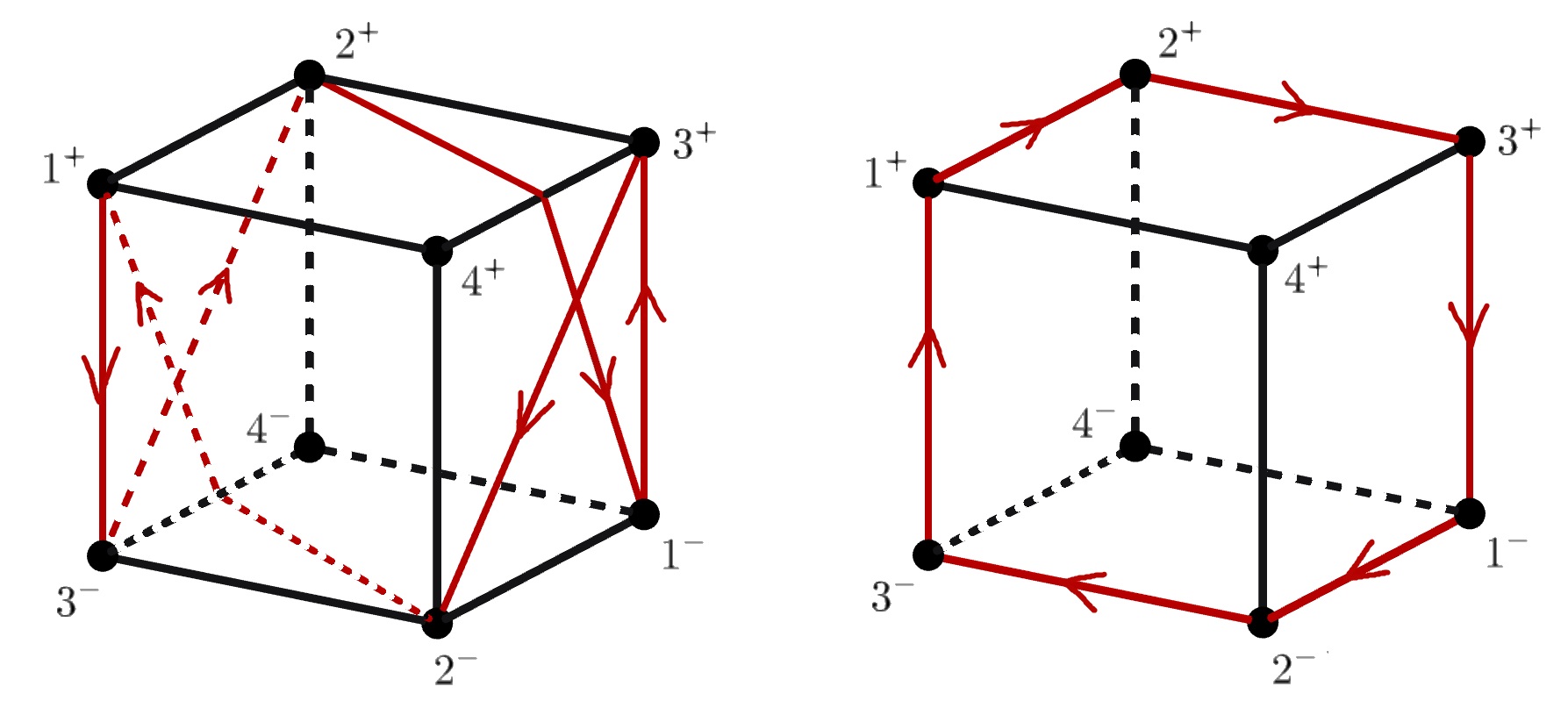}
    \caption{Two directed cycles on the surface of a cube. The second one is strictly shorter.}
    \label{fig1}
    \end{figure}
    
     In the set above, let $c$ be a cycle of the shortest length. We claim that the edges in $c$ do not cross on $\mathscr{S}$. Suppose that $u_1v_1$ and $u_2v_2$ are two directed edges that cross on $\mathscr{S}$, Without loss of generality, we assume that $c$ visits the four vertices in the order $u_1$, $v_1$, $u_2$, $v_2$. Then we can replace $u_1v_1$ and $u_2v_2$ with the directed edges $u_1u_2$ and $v_1v_2$, reverse the direction of the path from $v_1$ to $u_2$, and perform the same trick on their antipodal images. This will make $c$ a strictly shorter cycle satisfying the three conditions above, a contradiction. For example, in Figure \ref{fig1}, if we perform this trick on the cycle in the first picture, we obtain the cycle in the second picture, which is strictly shorter.

    Without loss of generality, suppose that $c$ visits the vertices of $\mathscr{S}$ in the order $1^+$, $2^+$, $\dots$, $(N-1)^+$, $1^-$, $2^-$, $\dots$, $(N-1)^-$, but not $N^+$ or $N^-$. For convenience, we denote the directed edges in $c$ by their starting points and endpoints, such as $1^+2^+$, $2^+3^+$, etc. We cut $\mathscr{S}$ along $c$ into two surfaces that are topologically equivalent to disks and are antipodal to each other. Let $\mathscr{S}^+$ be the surface containing $N^+$. We can cut $\mathscr{S}^+$ open by cutting along an edge inside $\mathscr{S}^+$ joining $N^+$ and a vertex of $c$. Without loss of generality, we assume that this line segment joins $(N-1)^-$ and $N^+$, and we denote it by $(N-1)^-N^+$. In this way, we obtain a polygon $P_{\mathscr{S}^+}$ that can be unfolded on the complex plane $\mathbb{C}$. In Figure \ref{fig2}, $\mathscr{S}$ is the surface of a cube with $N=4$, $\mathscr{S}^+$ is obtained by cutting $\mathscr{S}$ along the cycle from the previous example, and $P_{\mathscr{S}^+}$ is obtained after cutting $\mathscr{S}^+$ along $4^+3^-$ and unfolding it on the plane.

    \begin{figure}[h]
    \centering    
    \includegraphics[width=0.7\textwidth]{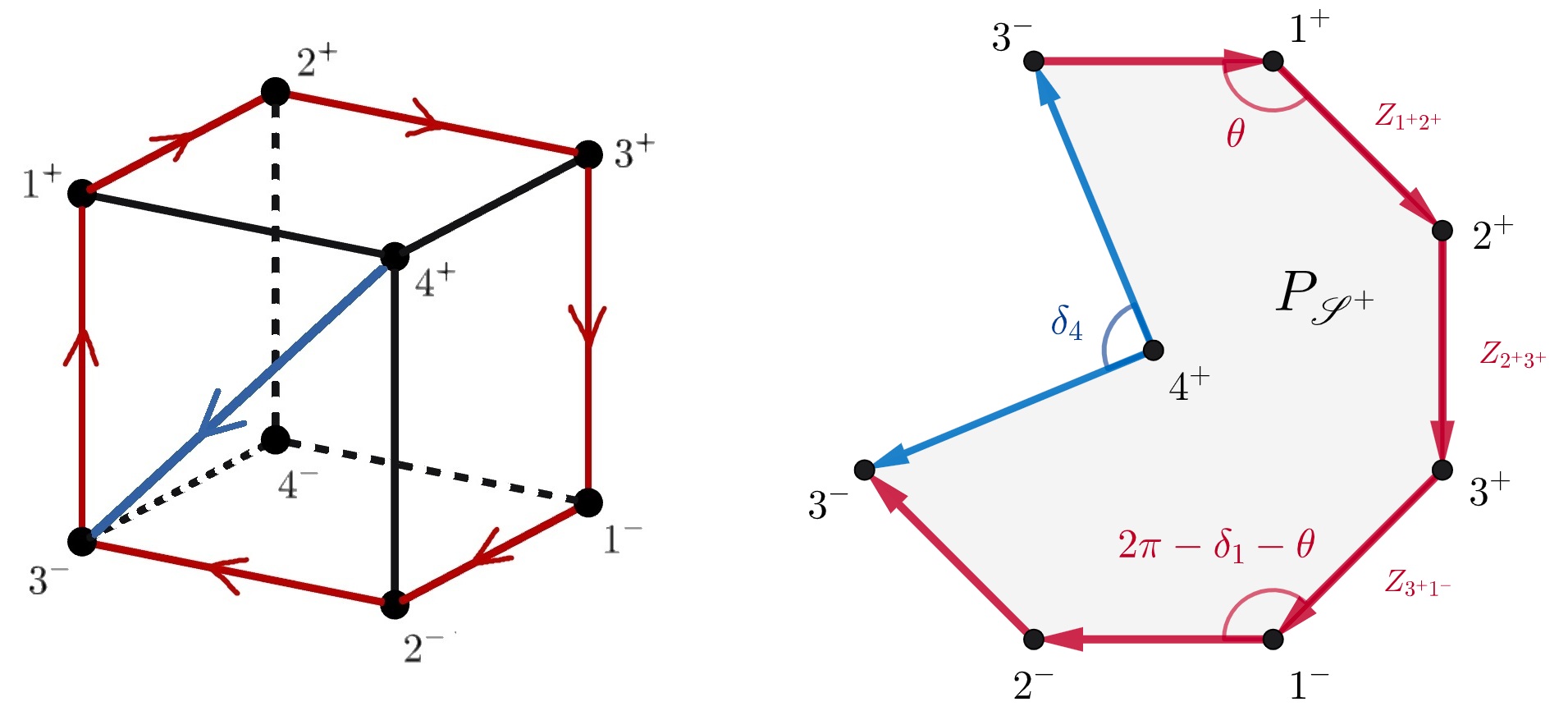}
    \caption{The surface $\mathscr{S}^+$ is cut open and unfolded to a planar polygon $P_{\mathscr{S}^+}$.}
    \label{fig2}
    \end{figure}

    Every directed edge of $P_{\mathscr{S}^+}$ determines a vector from its starting point to its endpoint. We think of these vectors as complex numbers, and denote them by $Z_{(N-1)^-1^+}$, $Z_{1^+2^+}$, $\dots$, $Z_{(N-2)^+(N-1)^+}$, $Z_{(N-1)^+1^-}$, $Z_{1^-2^-}$, $\dots$, $Z_{(N-2)^-(N-1)^-}$. Among these complex numbers, we take $N-1$ of them ranging from $Z_{1^+2^+}$ to $Z_{(N-1)^+1^-}$. By choosing an unfolding such that $Z_{(N-1)^-1^+}$ is a positive real number and $P_{\mathscr{S}^+}$ is on its right-hand side, these complex numbers are uniquely determined. In this way, we associate $\mathscr{S}$ to a vector $(Z_{1^+2^+}, Z_{2^+3^+}, \dots, Z_{(N-1)^+1^-})\in \mathbb{C}^{N-1}$. In Figure \ref{fig2}, the numbers $Z_{1^+2^+}$, $Z_{2^+3^+}$ and $Z_{3^+1^-}$ are labeled in the second picture.


    Conversely, the complex vector $(Z_{1^+2^+}$, $Z_{2^+3^+}$, $\dots$, $Z_{(N-1)^+1^-})$ also determines $P_{\mathscr{S}^+}$ and thus $\mathscr{S}$. First, the number $Z_{1^+2^+}$ determines the clockwise angle from $Z_{(N-1)^-1^+}$ to $Z_{1^+2^+}$, which we assume is $\pi-\theta$ (see Figure \ref{fig2} for an example). Then the clockwise angle from $Z_{(N-1)^+1^-}$ to $Z_{1^-2^-}$ equals $\pi-(2\pi-\delta_1-\theta)=\delta_1+\theta-\pi$, so $Z_{1^-2^-}$ is uniquely determined by $Z_{(N-1)^+1^-}$ and $\delta_1$. Therefore, given $Z_{1^+2^+}$, $Z_{2^+3^+}$, $\dots$, $Z_{(N-1)^+1^-}$ and the prescribed cone-deficits, we can similarly determine the remaining complex numbers up to $Z_{(N-2)^-(N-1)^-}$. The last two edges of $P_{\mathscr{S}^+}$ come from the cutting slit $(N-1)^-N^+$, so they have an equal modulus and make an angle of $\delta_N$. They are uniquely determined because the polygon $P_{\mathscr{S}^+}$ needs to close up. In particular, each of them is a complex linear combination of $Z_{1^+2^+}$, $Z_{2^+3^+}$, $\dots$, and $Z_{(N-1)^+1^-}$. 
    
    By varying the complex numbers $Z_{1^+2^+}$, $Z_{2^+3^+}$, $\dots$, $Z_{(N-1)^+1^-}$ locally, we can obtain polyhedral surfaces near $\mathscr{S}$, so we have constructed a local coordinate chart $\phi:\mathscr{U}\rightarrow \mathbb{C}^{N-1}$, where $\mathscr{U}$ is a neighborhood of $\mathscr{S}$ in $\mathcal{C}_{2N}(\delta_1, \delta_2, \dots, \delta_N)$. To show that $\mathcal{C}_{2N}(\delta_1, \delta_2, \dots, \delta_N)$ is a complex manifold, it remains to show that when two coordinate charts overlap, the change-of-coordinate map is a complex linear transformation.

    It would be helpful to think of $\phi$ as the restriction of a \textbf{developing map} $D: \widetilde{\mathscr{S}}\rightarrow\mathbb{C}$, where $\widetilde{\mathscr{S}}$ is the universal cover of $\mathscr{S}_0$, the complement of the vertices of $\mathscr{S}$. Let $p:\widetilde{\mathscr{S}}\rightarrow\mathscr{S}_0$ be the {covering map}. One can view $\widetilde{\mathscr{S}}$ as the space of homotopic paths in $\mathscr{S}_0$ with a common basepoint. If $e_1$ and $e_2$ are two directed edges (not including the endpoints) of $\widetilde{\mathscr{S}}$ such that $p(e_1)=p(e_1)$, then $D(e_1)$ and $D(e_2)$ can be computed by the analytic continuation of $\mathscr{U}$ along paths with different homotopy classes in $\mathscr{S}_0$. It turns out that $D(e_1)$ and $D(e_2)$ differ by a rotation by a constant angle that is determined by the prescribed cone-deficits. One may refer to the discussion on the developing map and the holonomy in \cite{Th2} (Section 3.5) or \cite{Th1} (the paragraph after Proposition 3.1) for more details.   


    Let $\phi':\mathscr{U}'\rightarrow \mathbb{C}^{N-1}$ be another chart such that $\mathscr{S}\in\mathscr{U}\cap\mathscr{U}'$. Suppose that $\phi'$ is constructed by cutting $\mathscr{S}$ along a different directed cycle $c'$ and a new cutting slit. Up to a composition of translations and rotations, we may assume that $\phi'$ is also a restriction of $D$. Let $e$ be a directed edge of $\widetilde{\mathscr{S}}$ so that $p(e)$ is in $c'$ or is the new cutting slit. If $p(e)$ is in $c$ or is $(N-1)^-N^+$, we have seen that $D(e)$ is a complex linear combination of $Z_{1^+2^+}$, $Z_{2^+3^+}$, $\dots$, $Z_{(N-1)^+1^-}$. If $p(e)$ is any other edge of $\mathscr{S}$, then one can find edges $e_1$, $e_2$, $\dots$, $e_n$ in $\widetilde{\mathscr{S}}$ that form a cycle in $\widetilde{\mathscr{S}}$ with $e$, and $p(e_i)$ (or $-p(e_i)$) is $(N-1)^-N^+$ or in $c$. Then $D(e)$ is again a complex linear combination of $Z_{1^+2^+}$, $Z_{2^+3^+}$, $\dots$, $Z_{(N-1)^+1^-}$. Thus, the transformation from $Z_{1^+2^+}$, $Z_{2^+3^+}$, $\dots$, $Z_{(N-1)^+1^-}$ to the coordinates given by $\phi'$ is complex linear, and is valid in $\mathscr{U}\cap\mathscr{U}'$. This shows that $\mathcal{C}_{2N}(\delta_1, \delta_2, \dots, \delta_N)$ is a complex manifold of dimension $N-1$.
    \end{proof}

    In the proof of Theorem \ref{thm1}, we use $N-1$ directed edges of $\mathscr{S}$ to construct a local coordinate chart containing $\mathscr{S}$. We have seen that the coordinates depend on which $N-1$ directed edges of $\mathscr{S}$ we choose (also on how we develop them onto the plane). The choice of directed edges does not have to be restricted to the way in the proof of Theorem \ref{thm1}. In fact, if $e_1$, $e_2$, $\dots$, $e_{N-1}$ are directed edges of $\widetilde{\mathscr{S}}$ such that the complex linear transformation from $(Z_{1^+2^+}$, $Z_{2^+3^+}$, $\dots$, $Z_{(N-1)^+1^-})$ to $(D(e_1), D(e_2), \dots, D(e_{N-1}))$ is invertible, then we can also use $(D(e_1), D(e_2), \dots, D(e_{N-1}))$ as coordinates to describe surfaces near $\mathscr{S}$. In this case, we call $(D(e_1), D(e_2), \dots, D(e_{N-1}))$ a \textbf{local frame} near $\mathcal{S}$ on $\mathcal{C}_{2N}(\delta_1, \delta_2, \dots, \delta_N)$. For example, in Figure \ref{fig3}, we draw three different sets of directed edges on the surface of a cube. By choosing a way to develop each set of directed edges onto the plane, we can get different local frames near this surface on $\mathcal{C}_{8}(\frac{\pi}{2}, \frac{\pi}{2}, \frac{\pi}{2}, \frac{\pi}{2})$. We will return to local frames when we discuss parallelogram decompositions of polyhedral surfaces in Section 4.

    \begin{figure}[h]
    \centering    \includegraphics[width=0.8\textwidth]{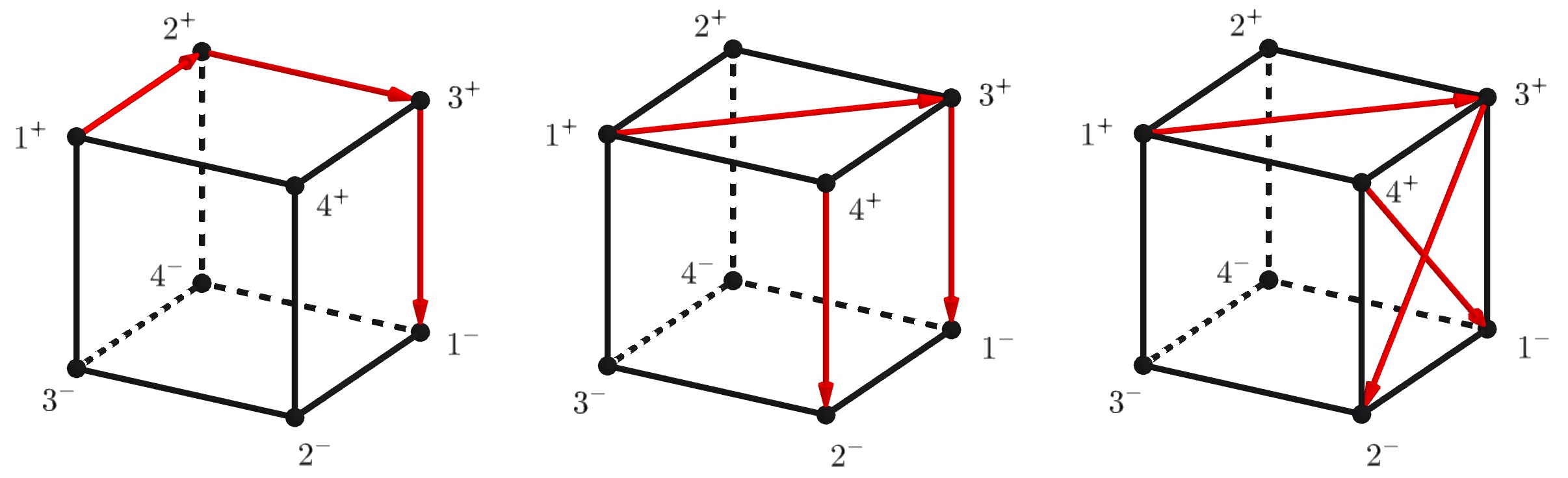}
    \caption{Directed edges on the surface of a cube that give rise to different local frames.}
    \label{fig3}
    \end{figure}


    Next, let $\mathcal{M}_{2N}(\delta_1, \delta_2, \dots, \delta_N)$ be the moduli space of $\mathcal{C}_{2N}(\delta_1, \delta_2, \dots, \delta_N)$, in which two polyhedral surfaces are equivalent if and only if they differ by a similarity preserving the vertex-labels. Alternatively, we can regard $\mathcal{M}_{2N}(\delta_1, \delta_2, \dots, \delta_N)$ as the space of polyhedral surfaces in $\mathcal{C}_{2N}(\delta_1, \delta_2, \dots, \delta_N)$ whose surface areas are equal to $1$. 

    In \cite{Th1}, it was shown that the moduli space of convex polyhedra with $N$ prescribed cone-deficits has the structure of a complex hyperbolic manifold of dimension $N-3$. The main idea of proof is to show that the surface area is a Hermitian form of signature $(1,n-3)$ with respect to the coordinates arising from decomposing the surfaces into triangles. The argument is based on induction on the number of vertices, where multiple vertices of a surface may ``collide'' to produce a surface with fewer vertices. One may refer to \cite{Th1} and \cite{Sch} for more details of this operation. 
    
    Next, we are going to use collisions of vertices to prove the following result:


\begin{theorem} \label{thm2}
    When $N\geq 3$, the space $\mathcal{M}_{2N}(\delta_1, \delta_2, \dots, \delta_N)$ has the structure of a real hyperbolic manifold of dimension $2N-3$.
\end{theorem}
\begin{proof}
    The idea of proof is to express the surface area as a quadratic form of signature $(1,2N-3)$ on $\mathcal{C}_{2N}(\delta_1, \delta_2, \dots, \delta_N)$. However, the coordinates we use will be based on induction and different from those in Theorem \ref{thm1}.

    In the base case where $N=3$, $\mathcal{M}_{6}(\delta_1, \delta_2, \delta_3)$ is the moduli space of centrally symmetric octahedra. We showed in \cite{Wang} that the surface area is a quadratic form of signature $(1,3)$ with respect to the coordinates obtained by decomposing the surfaces into parallelograms. 

    When $N>3$, let $\mathscr{S}$ be a polyhedral surface in $\mathcal{C}_{2N}(\delta_1, \delta_2, \dots, \delta_N)$. We choose two non-antipodal vertices such that the sum of their cone-deficits is strictly less than $2\pi$. This is possible since $\sum\limits_{i=1}^N\delta_i=2\pi$. Without loss of generality, we assume that these two vertices are $1^+$ and $2^+$. Then we choose an edge connecting $1^+$ and $2^+$, cut $\mathscr{S}$ open along this edge, and glue a cone to the slit. This cone can be obtained by gluing two congruent triangles with base angles $\frac{\delta_1}{2}$ and $\frac{\delta_2}{2}$. In this way, we kill the cone-deficits at $1^+$ and $2^+$, and produce a new vertex $V^+$ with cone-deficit $\frac{\delta_1+\delta_2}{2}$. This operation is demonstrated in Figure \ref{fig4}. Simultaneously, We perform the same operation on the vertices $1^-$ and $2^-$. In this way, we construct a new polyhedral surface $\mathscr{S}'$ in $\mathcal{C}_{2N-2}(\delta_1+\delta_2, \delta_3,\dots, \delta_N)$. By induction hypothesis, there exists a local coordinate chart from an open neighborhood of  $\mathscr{S}'$ in $\mathcal{C}_{2N-2}(\delta_1+\delta_2, \delta_3,\dots, \delta_N)$ to $\mathbb{R}^{2N-4}$ with respect to which we can write the surface area function in the form $k_1x_1^2-k_2x_2^2-\dots-k_{2N-4}x_{2N-4}^2$. 

    \begin{figure}[h]
    \centering    \includegraphics[width=0.8\textwidth]{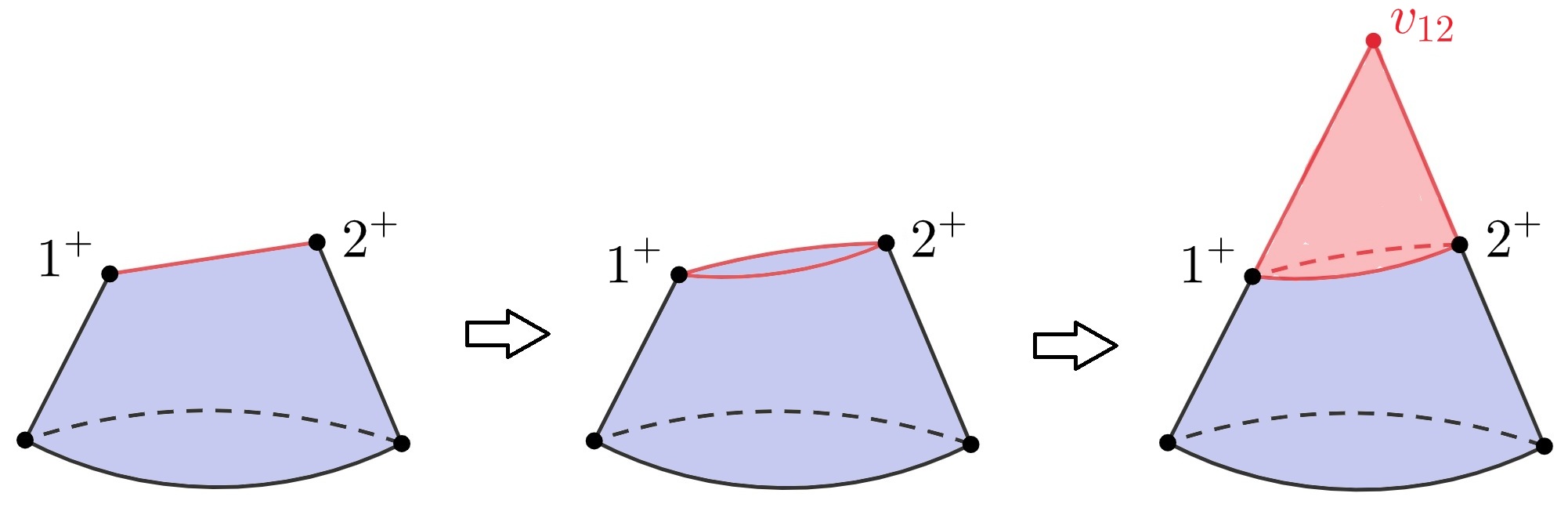}
    \caption{Vertices $1^+$ and $2^+$ ``collide'' to $v_{12}$, resulting in a surface with one fewer vertex.}
    \label{fig4}
    \end{figure}

    The surface area of $\mathscr{S}$ equals that of $\mathscr{S}'$ minus the areas of the four congruent triangles with base angles $\frac{\delta_1}{2}$ and $\frac{\delta_2}{2}$. To complete the induction, it suffices to show that the area of each triangle has the form $k(x_{2N-3}^2+x_{2N-2}^2)$, where $k>0$, and $x_1$, $x_2$, $\dots$, $x_{2N-4}$, $x_{2N-3}$, $x_{2N-2}$ form a set of real coordinates for surfaces near $\mathscr{S}$. 

    Fix a ray $r$ emanating from $v_{12}$ on $\mathscr{S}'$ as a reference. The surface $\mathscr{S}$ is uniquely determined by both the distance $d$ between $v_{12}$ and $1^+$, and the clockwise angle $\theta$ from the ray $r$ to the ray from $v_{12}$ to $1^+$. Then the area of each triangle is a multiple of $d^2$ by a constant determined by $\frac{\delta_1}{2}$ and $\frac{\delta_2}{2}$. Thus, we can take $x_{2N-3}=d\cos{\theta}$ and $x_{2N-2}=d\sin{\theta}$. The coordinates $x_1$, $x_2$, $\dots$, $x_{2N-4}$, $x_{2N-3}$, $x_{2N-2}$ remain valid for surfaces near $\mathscr{S}$ in $\mathcal{C}_{2N}(\delta_1, \delta_2, \dots, \delta_N)$. This completes the induction.    
    
    Therefore, surfaces near $\mathscr{S}$ in $\mathcal{C}_{2N}(\delta_1, \delta_2, \dots, \delta_N)$ with area $1$ correspond to vectors in $\mathbb{R}^{2N-2}$ of length $1$ with respect to a quadratic form of signature $(1, 2N-3)$. This establishes local coordinate charts from $\mathcal{M}_{2N}(\delta_1, \delta_2, \dots, \delta_N)$ to the real hyperbolic space $\mathbb{H}^{2N-3}$. Finally, any change-of-coordinate map on $\mathcal{C}_{2N}(\delta_1, \delta_2, \dots, \delta_N)$ keeps the surface area invariant. This shows that $\mathcal{M}_{2N}(\delta_1, \delta_2, \dots, \delta_N)$ is a real hyperbolic manifold of dimension $2N-3$.
\end{proof}

Let $\overline{\mathcal{M}_{2N}}(\delta_1, \delta_2, \dots, \delta_N)$ be the metric completion of $\mathcal{M}_{2N}(\delta_1, \delta_2, \dots, \delta_N)$ with respect to the real hyperbolic metric above. For example, the surface in the third picture of Figure \ref{fig4} is not in $\mathcal{M}_{2N}(\delta_1, \delta_2, \dots, \delta_N)$ but in  $\overline{\mathcal{M}_{2N}}(\delta_1, \delta_2, \dots, \delta_N)$ if it has area $1$. This is because one can construct a Cauchy sequence of surfaces in $\mathcal{M}_{2N}(\delta_1, \delta_2, \dots, \delta_N)$ converging to it by colliding two vertices and normalizing the surface area at the same time. We will return to this space in Section 4.

\section{Loop Arrangements}

In this section, we show that there is a natural correspondence between a parallelogram decomposition of a centrally symmetric polyhedral surface and a ``loop arrangement'' on the sphere $S^2$.

Suppose that we are given $2N$ distinct vertices on $S^2$ labeled by $1^+$, $1^-$, $2^+$, $2^-$, $\dots$, $N^+$ and $N^-$, where $i^+$ and $i^-$ are antipodal for $1\leq i\leq N$. By a \textbf{loop} we mean a great circle on $S^2$ that does not pass through any labeled vertices. A set of loops form a \textbf{loop arrangement} if any two of them are non-homotopic with respect to the labeled vertices. Since loops are great circles, every two of them intersect exactly twice at two antipodal points. We assume without loss of generality that no three loops are concurrent, otherwise we can perturb any one of the them. As an example, in Figure \ref{fig5}, we sketch the ``front view'' of a loop arrangement with $6$ loops on a sphere with $8$ labeled vertices. Each red curve represents a loop. We use black dots to represent the vertices $1^+$, $2^+$, $3^+$, and $4^+$. In addition, we use gray dots to represent the vertices $2^-$, $3^-$, and $4^-$ (even though they should be invisible from the actual front view) because it will be helpful for our demonstrations later. 

\begin{figure}[h]
    \centering    \includegraphics[width=0.24\textwidth]{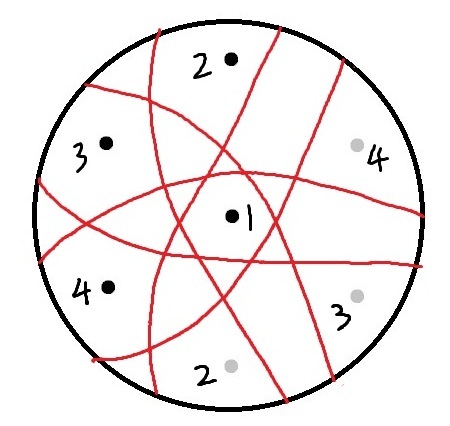}
    \caption{A loop arrangement with $6$ loops (red) on a sphere with $8$ labeled vertices.}
    \label{fig5}
\end{figure}

Let $\mathscr{A}$ be a loop arrangement on a sphere with $2N$ labeled vertices. Regarding it as a graph on the sphere, we can take its dual graph $\mathscr{D}^*$. Since no three loops are concurrent, the faces of $\mathscr{D}^*$ are all quadrilaterals. As we show next, $\mathscr{D}^*$ is closely related to the decomposition of polyhedral surfaces in $\mathcal{C}_{2N}(\delta_1, \delta_2, \dots, \delta_N)$ into parallelograms. More precisely, we will first turn each quadrilateral face into a parallelogram by specifying its side lengths and angles, and then check that we do obtain a polyhedral surface in $\mathcal{C}_{2N}(\delta_1, \delta_2, \dots, \delta_N)$. 

We first label all loops in $\mathscr{A}$. For example, if $\mathscr{A}$ is the loop arrangement in Figure \ref{fig5}, we can denote its $8$ loops by $l_a$, $l_b$, $l_c$, $l_d$, $l_e$, and $l_f$. For convenience, these loops are marked by $a$, $b$, $\dots$, $f$ in the first picture of Figure \ref{fig6}. In the second picture, we draw the dual graph $\mathscr{D}^*$ of $\mathscr{A}$ in blue.

\begin{figure}[h]
    \centering    \includegraphics[width=0.65\textwidth]{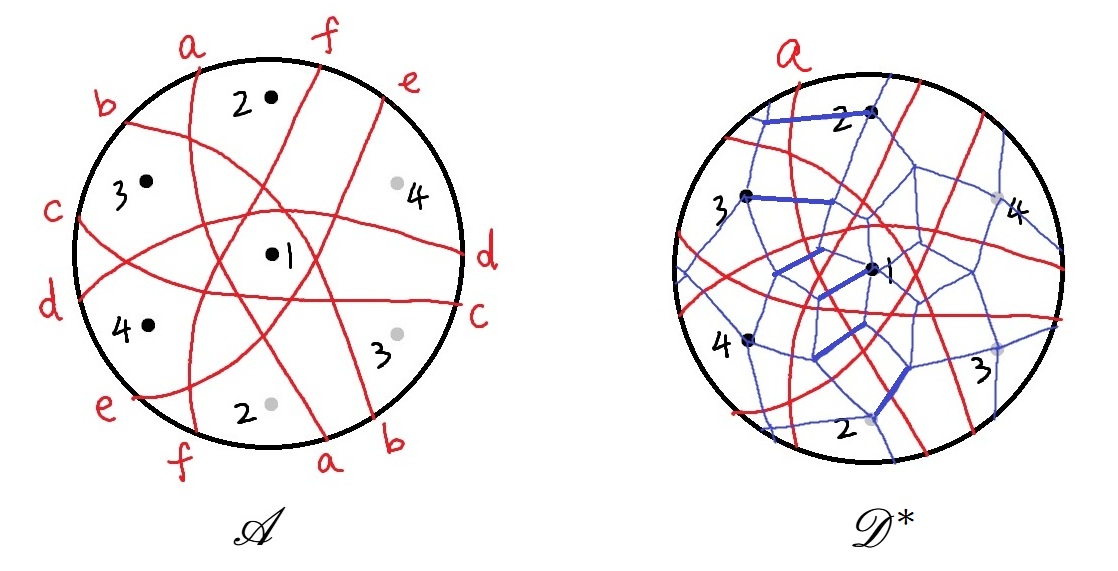}
    \caption{The loop arrangement $\mathscr{A}$ with its dual graph $\mathscr{D}^*$ (blue).}
    \label{fig6}
\end{figure}

 We say that an edge in $\mathscr{D}^*$ is \textbf{transverse to} the loop $l_i$ if it is dual to an arc of this loop. For each loop $l_i$, We assign a positive number $i$ to all edges in $\mathscr{D}^*$ transverse to $l_i$. These numbers will be the edge-lengths of the quadrilaterals. For example, in Figure \ref{fig6}, the thickened edges are all transverse to $l_a$ and assigned length $a$.

Next, we assign angles to each quadrilateral. Let $Q$ be a quadrilateral in $\mathscr{D}^*$. Consider two edges of $Q$ that meet at a vertex $v$ in $Q$. We are going to assign a value $\theta$ as the angle made by these two edges in $Q$. To compute $\theta$, let $l_i$ and $l_j$ be the two loops to which these two edges are transverse. Then $l_i$ and $l_j$ divide the sphere into $4$ lunes, one of which contains $v$ in its interior. We denote this lune by $\mathcal{L}$. For example, in the first picture of Figure \ref{fig7}, we sketch a quadrilateral $Q$ (blue), the angle $\theta$ we want to compute in $Q$, and the lune $\mathcal{L}$ (red) we consider to compute $\theta$.

\begin{figure}[h]
    \centering    \includegraphics[width=0.7\textwidth]{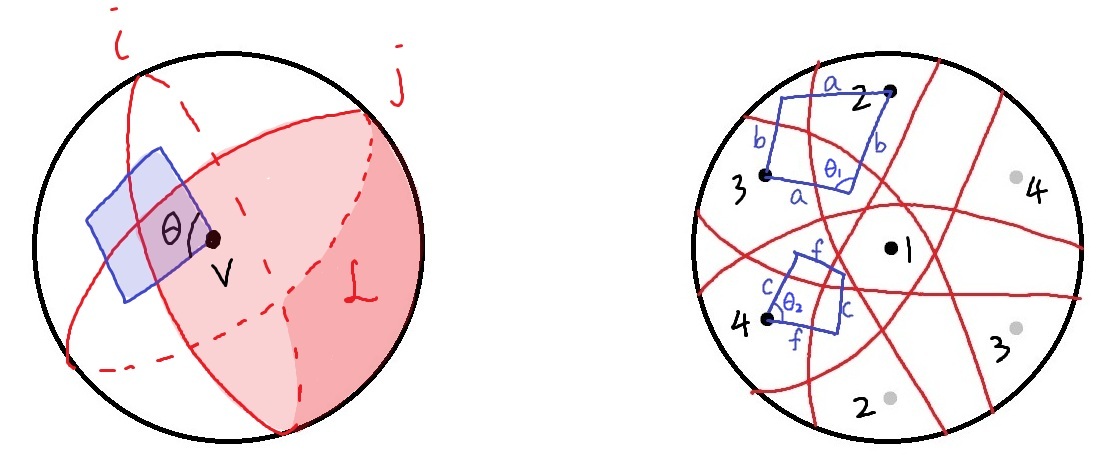}
    \caption{Assigning angles to quadrilaterals in $\mathscr{D}^*$.}
    \label{fig7}
\end{figure}

Consider all the labeled vertices enclosed by $\mathcal{L}$. We compute $\theta$ by first taking the sum of the cone-deficits associated to the enclosed labeled vertices, then dividing this angle sum by half, and finally taking its supplementary angle. That is,

\begin{equation}\label{eq1}
    \theta=\pi-\frac{1}{2}\sum_{n^+\text{ or } n^-\in\mathcal{L}} \delta_n
\end{equation}

For example, consider the two quadrilaterals in the second picture of Figure \ref{fig7}. For the top quadrilateral, we have $\theta_1=\pi-\frac{1}{2}\delta_1$, since $\mathcal{L}$ encloses $1^+$ only. For the bottom quadrilateral, we have $\theta_2=\pi-\frac{1}{2}(\delta_1+\delta_4)=\frac{1}{2}\delta_2+\frac{1}{2}\delta_3$.

By symmetry, the value of the opposite angle in $Q$ is also $\theta$. As the sum of all $\delta_i$'s equals $2\pi$, Equation \ref{eq1} implies that $\theta>0$. Let $\theta'$ be the value assigned to any one of the remaining two angles in $Q$, and let $\mathcal{L}'$ be the lune we consider to compute $\theta'$. Note that $\mathcal{L}\cup\mathcal{L}'$ is a hemisphere. Therefore, the sum of the cone-deficits of the labeled vertices enclosed in $\mathcal{L}$ and those enclosed in $\mathcal{L}'$ equals $2\pi$. Applying Equation \ref{eq1}, we get $\theta+\theta'=\pi$. Since $\theta, \theta'>0$, with the edge-lengths we assigned previously, $Q$ is now a parallelogram. 

Therefore, we regard all quadrilaterals in $\mathscr{D}^*$ as solid parallelograms. In this way, we obtain a polyhedral surface (Euclidean cone sphere) that is almost everywhere flat except possibly at the vertices of the parallelograms. These parallelograms form a \textbf{parallelogram decomposition} of the surface. By varying the edge-lengths of the parallelograms, we obtain different polyhedral surfaces. Our next goal is to show that these surfaces all belong to $\mathcal{C}_{2N}(\delta_1, \delta_2, \dots, \delta_N)$. That is:

\begin{theorem} \label{thm3}
    On a sphere $S^2$ with $2N$ labeled vertices, let $\mathscr{A}$ be a loop arrangement with $2N-2$ loops ($N\geq 3$). Let $\mathscr{D}^*$ be the dual graph of $\mathscr{A}$, and $\mathscr{S}$ be any surface arising from $\mathscr{D}^*$ by assigning edge-lengths and angles to the quadrilaterals in $\mathscr{D}^*$ in the way discussed above. Then $\mathscr{S}\in \mathcal{C}_{2N}(\delta_1, \delta_2, \dots, \delta_N)$.
\end{theorem}

\begin{proof}
    The key step of proof is to establish Equation \ref{eq2} below and generalize it to Equation \ref{eq4} by induction. Then we will use Equation \ref{eq4} to check that every vertex in $\mathscr{S}$ has the correct cone-deficit. 

    Since $2N-2\geq 4$, we can choose three distinct loops in $\mathscr{A}$, denoted by $l_i$, $l_j$ and $l_k$. By assumption, they are not concurrent, so they divide the sphere into $8$ geodesic triangles. Denote one of these geodesic triangles by $\mathcal{T}$. Let $u_i$, $u_j$, and $u_k$ be the vertices of $\mathcal{T}$ such that $u_i$ is opposite to the edge of $\mathcal{T}$ contained in $l_i$, and similarly for $u_j$ and $u_k$. The first picture of Figure \ref{fig8} is a sketch of $\mathcal{T}$ with its vertices.

    \begin{figure}[h]
    \centering    \includegraphics[width=0.95\textwidth]{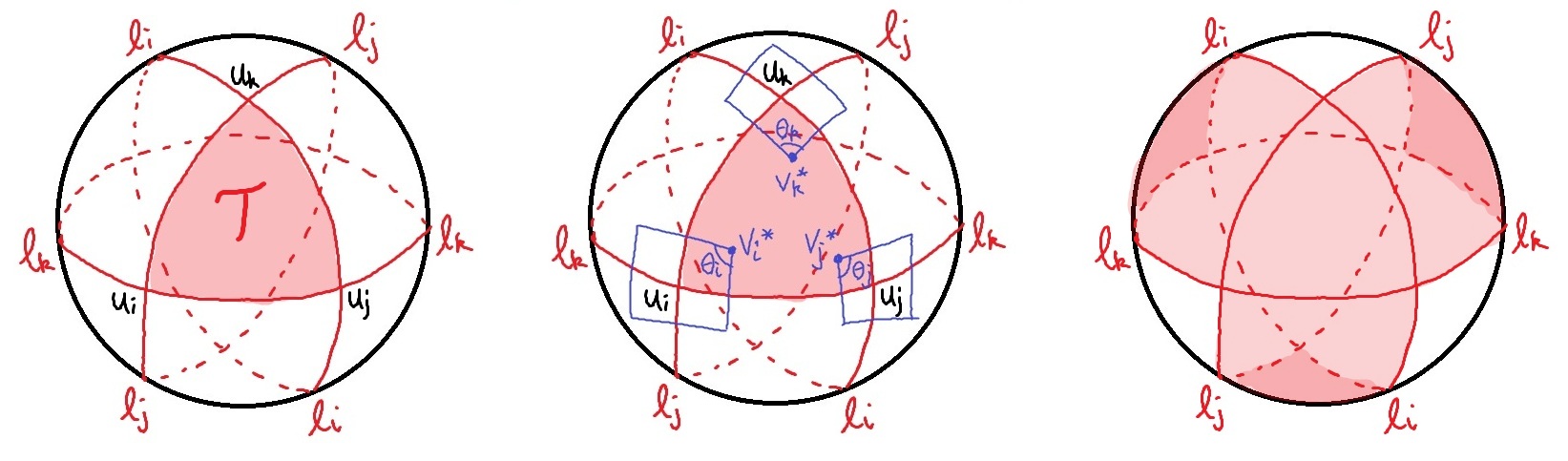}
    \caption{1) The geodesic triangle $\mathcal{T}$; 2) The vertices $v^*_i$, $v^*_j$, and $v^*_k$; 3) The region $\mathcal{L}_i\cup\mathcal{L}_j\cup\mathcal{L}_k$.}
    \label{fig8}
    \end{figure}

     Let $p^*_i$ be the parallelogram dual to $u_i$ in $\mathscr{D}^*$, and $v^*_i$ be its unique vertex in the interior of $\mathcal{T}$. Let $\theta_i$ be the angle of $p^*_i$ at $v^*_i$. Similarly, we define $p^*_j$, $p^*_k$, $v^*_j$, $v^*_k$, $\theta_j$ and $\theta_k$. The second picture of Figure \ref{fig8} shows an example where $v^*_i$, $v^*_j$, and $v^*_k$ are mutually distinct, but they may also repeat in some cases.
    
    Now we want to establish the following Equation \ref{eq2}:
    \begin{equation} \label{eq2}
        \theta_i+\theta_j+\theta_k=2\pi-\sum\limits_{n^+\text{ or } n^-\in\mathcal{T}} \delta_n
    \end{equation}

    Let $\mathcal{L}_i$ be the lune bounded by $l_j$ and $l_k$ and containing $\mathcal{T}$, and similarly define $\mathcal{L}_j$ and $\mathcal{L}_k$. Consider the region $\mathcal{L}_i\cup\mathcal{L}_j\cup\mathcal{L}_k$. Note that its complement in $S^2$ is just the interior of its image under the antipodal map. In the third picture of Figure \ref{fig8}, the region $\mathcal{L}_i\cup\mathcal{L}_j\cup\mathcal{L}_k$ is painted in red and its complement in $S^2$ is white. In addition, $\mathcal{L}_i\cap\mathcal{L}_j\cap\mathcal{L}_k=\mathcal{T}$. Thus, by symmetry and the Gauss-Bonnet Theorem, we have 

    \begin{equation} \label{eq3}
        \sum\limits_{n^+\text{ or } n^-\in\mathcal{L}_i}\delta_n+\sum\limits_{n^+\text{ or } n^-\in\mathcal{L}_j}\delta_n+\sum\limits_{n^+\text{ or } n^-\in\mathcal{L}_k}\delta_n-2\sum\limits_{n^+\text{ or } n^-\in\mathcal{T}}\delta_n=2\pi
    \end{equation}

    Therefore, applying Equation \ref{eq1} and Equation \ref{eq3}, we have

    \begin{align*}
        \theta_i+\theta_j+\theta_k &=(\pi-\frac{1}{2}\sum_{n^+\text{ or } n^-\in\mathcal{L}_i} \delta_n)+(\pi-\frac{1}{2}\sum_{n^+\text{ or } n^-\in\mathcal{L}_j} \delta_n)+(\pi-\frac{1}{2}\sum_{n^+\text{ or } n^-\in\mathcal{L}_k} \delta_n) \\
        &= 3\pi-\frac{1}{2}(2\pi+2\sum\limits_{n^+\text{ or } n^-\in\mathcal{T}}\delta_n) \\
        &=2\pi-\sum\limits_{n^+\text{ or } n^-\in\mathcal{T}} \delta_n
    \end{align*}

    We have thus derived Equation \ref{eq2}.

    Now we focus on all the loops in $\mathscr{A}$. These loops divide $S^2$ into convex geodesic polygons. Let $\mathcal{P}_m$ be one of these polygons. Without loss of generality, we assume that it is bounded by the loops $l_1$, $l_2$, $\dots$, and $l_m$ in $\mathscr{A}$. Figure \ref{fig9} shows an example when $m=6$. Consider the $m$ parallelograms in $\mathscr{D}^*$ dual to the $m$ vertices of $\mathcal{P}_m$. Let $v^*_1$, $v^*_2$, $\dots$, $v^*_m$ be the vertices of these parallelograms inside $\mathcal{P}_m$, which may be distinct or repeat. Let $\Theta_m$ be the sum of the $m$ angles of these parallelograms at the $v^*_i$'s. In the example in Figure \ref{fig9}, the $6$ parallelograms are drawn in blue, and the sum of all the angles marked in red equals $\Theta_6$.

    \begin{figure}[h]
    \centering    \includegraphics[width=0.65\textwidth]{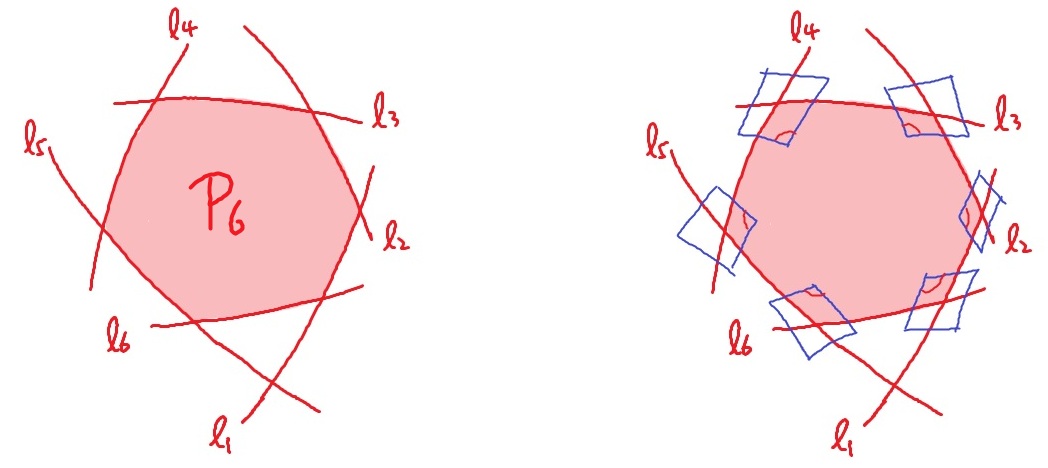}
    \caption{An example of $\mathcal{P}_6$. The sum of all the angles marked in red equals $\Theta_6$.}
    \label{fig9}
    \end{figure}
    
    Next, we want to establish the following equation:

    \begin{equation} \label{eq4}
        \Theta_m=2\pi-\sum\limits_{n^+\text{ or } n^-\in\mathcal{P}_m} \delta_n
    \end{equation}

    Since the base case is Equation \ref{eq2}, it remains to complete the inductive step.

    Consider the convex geodesic $(m-1)$-gon made by $l_1$, $l_2$, $\dots$, $l_{m-1}$ that contains $\mathcal{P}_m$. We denote this polygon by $\mathcal{P}_{m-1}$, and let $\mathcal{T}$ be the triangle $\mathcal{P}_{m-1}\setminus\mathcal{P}_{m}$. In the previous example of $\mathcal{P}_6$, $\mathcal{T}$ is the green triangle in Figure \ref{fig10}. Let $\Theta_{3}$ be the sum of the $3$ angles for $\mathcal{T}$ as we defined previously, and let $\Theta_{m-1}$ be the sum of the $m-1$ angles for $\mathcal{P}_{m-1}$. In the first picture of Figure \ref{fig10}, $\Theta_{3}$ is the sum of the three angles marked in green. In the second picture, $\Theta_{m-1}$ is the sum of the five angles marked in blue.

    \begin{figure}[h]
    \centering    \includegraphics[width=0.9\textwidth]{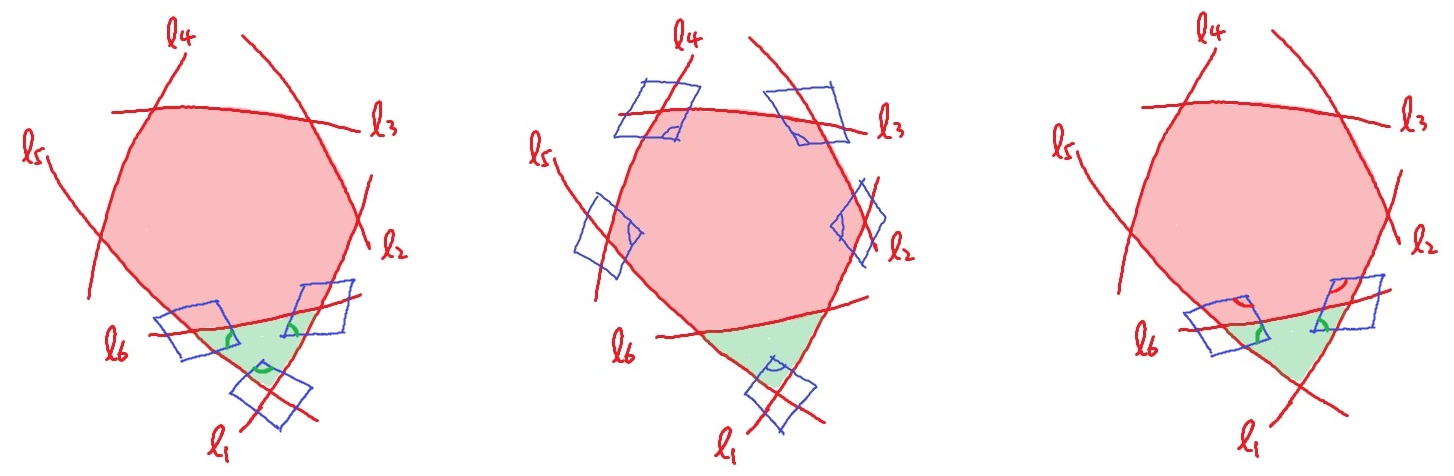}
    \caption{Computing $\Theta_{m}$ using $\Theta_{3}$ (sum of the green angles) and $\Theta_{m-1}$ (sum of the blue angles).}
    \label{fig10}
    \end{figure}
    
    By induction hypothesis, we have
    
    $$\Theta_{m-1}=2\pi-\sum\limits_{n^+\text{ or } n^-\in\mathcal{P}_{m-1}} \delta_n$$

    Now we add up the $m$ angles for $\mathcal{P}_m$ (whose sum equals $\Theta_m$) and the $3$ angles for $\mathcal{T}$ (whose sum is $2\pi-\sum\limits_{n^+\text{ or } n^-\in\mathcal{T}} \delta_n$ by Equation \ref{eq2}). Among these angles, there are two pairs of complementary angles, so their sum equals $2\pi$. These four angles appear in the two parallelograms dual to the two vertices that belong to both $\mathcal{P}_m$ and $\mathcal{T}$. In Figure \ref{fig10}, we draw these two parallelograms and mark the four angles in the third picture. Apart from these four angles, the sum of the remaining angles equals $\Theta_{m-1}$. Therefore, we obtain $\Theta_{m}+\Theta_{3}=\Theta_{m-1}+2\pi$.

    Combined with the induction hypothesis, we have

       $$ \Theta_m+2\pi-\sum\limits_{n^+\text{ or } n^-\in\mathcal{T}} \delta_n=\Theta_{m-1}+2\pi
        =2\pi-\sum\limits_{n^+\text{ or } n^-\in\mathcal{P}_{m-1}} \delta_n+2\pi$$

    After rearrangement, it becomes

        $$\Theta_m=2\pi-(\sum\limits_{n^+\text{ or } n^-\in\mathcal{P}_{m-1}} \delta_n-\sum\limits_{n^+\text{ or } n^-\in\mathcal{T}} \delta_n)=2\pi-\sum\limits_{n^+\text{ or } n^-\in\mathcal{P}_{m}} \delta_n$$
   
    where the second equality follows from the fact that $\mathcal{P}_{m-1}=\mathcal{P}_m\cup\mathcal{T}$. This completes the induction and the proof of Equation \ref{eq4}. 
            
    Finally, we check that every vertex in $\mathscr{S}$ has the right cone-deficit. Let $v^*$ be any vertex in $\mathscr{D}^*$. We need to show that the cone-angle at $v^*$ in $\mathscr{S}$ is $2\pi-\delta_n$ if $v^*$ is the labeled vertex $n^+$ or $n^-$, and $2\pi$ otherwise. 

    Let $\mathcal{P}_{m}$ be the convex geodesic polygon cut by the loops in $\mathscr{A}$ that is dual to $v^*$. Then the parallelograms dual to the vertices of $\mathcal{P}_{m}$ meet at $v^*$, so $\Theta_m$ is just the cone-angle at $v^*$. In addition, $v^*$ is the only vertex in $\mathscr{D}^*$ that is inside $\mathcal{P}_{m}$. Thus, by Equation \ref{eq4}, if $v^*$ is the labeled vertex $n^+$ or $n^-$, then $\Theta_m=2\pi-\delta_n$. If it is not a labeled vertex, then the sum of cone-deficits inside $\mathcal{P}_{m}$ is $0$, so $\Theta_m=2\pi$. This completes the proof of the theorem. 
\end{proof}

In preparation for the next section, we introduce some terminologies. 

Recall that all quadrilaterals in $\mathscr{D}^*$ are parallelograms whose angles are determined by the prescribed cone-deficits. We can assign arbitrary positive numbers to the edge-lengths of those parallelograms, provided that all the edges transverse to the same loop have the same length. In this way, we obtain a polyhedral surface. Let $\mathcal{O}$ be the collection of all polyhedral surfaces that can be obtained in this way. By Theorem \ref{thm3}, $\mathcal{O}$ is a subset of $ \mathcal{C}_{2N}(\delta_1, \delta_2, \dots, \delta_N)$. 

If we allow some edges in $\mathscr{D}^*$ to have $0$ length, what we can obtain is a larger set denoted by $\overline{\mathcal{O}}$. Note that $\overline{\mathcal{O}}\nsubseteq \mathcal{C}_{2N}(\delta_1, \delta_2, \dots, \delta_N)$, since we can assign length $0$ to all edges in $\mathscr{D}^*$. From now on, we will call $\mathcal{O}$ the \textbf{space arising from $\mathscr{A}$}, and $\overline{\mathcal{O}}$ \textbf{the closure of $\mathbf{\mathcal{O}}$}. By a \textbf{face} $\sigma_i$ of $\overline{\mathcal{O}}$ we mean the collection of elements in $\overline{\mathcal{O}}$ where the lengths of all edges in $\mathscr{D}^*$ transverse to the loop $l_i$ are $0$. 

Now suppose that $\mathscr{A}'$ is also a loop arrangement on a sphere with $2N$ labeled vertices. We say that $\mathscr{A}$ and $\mathscr{A}'$ are \textbf{equivalent} if there exists a homeomorphism between the two spheres that maps vertices to vertices and loops to loops preserving their labels. If $\mathscr{A}$ and $\mathscr{A}'$ differ by only one loop (that is, they become equivalent after we replace one loop in $\mathscr{A}$ or $\mathscr{A}'$ by another), we say that $\mathscr{A}$ and $\mathscr{A}'$ are \textbf{adjacent}. 
 
 For example, consider the loop arrangements $\mathscr{A}$ and $\mathscr{A}'$ in Figure \ref{fig11}. Note that $\mathscr{A}'$ can be obtained from $\mathscr{A}$ by moving the vertices $2^+$, $3^+$, $2^-$ and $3^-$ across the loop $l_a$ simultaneously. In fact, $\mathscr{A}$ and $\mathscr{A}'$ are adjacent. One can see this by dragging the vertices $2^+$, $3^+$, $2^-$ and $3^-$ in $\mathscr{A}'$ to their positions in $\mathscr{A}$, and distorting the loop $l_a$ in $\mathscr{A}'$ (without changing its homotopy class) to the green curve in the third picture. This shows that $\mathscr{A}$ and $\mathscr{A}'$ differ only by $l_a$.

    \begin{figure}[h]
    \centering    \includegraphics[width=0.95\textwidth]{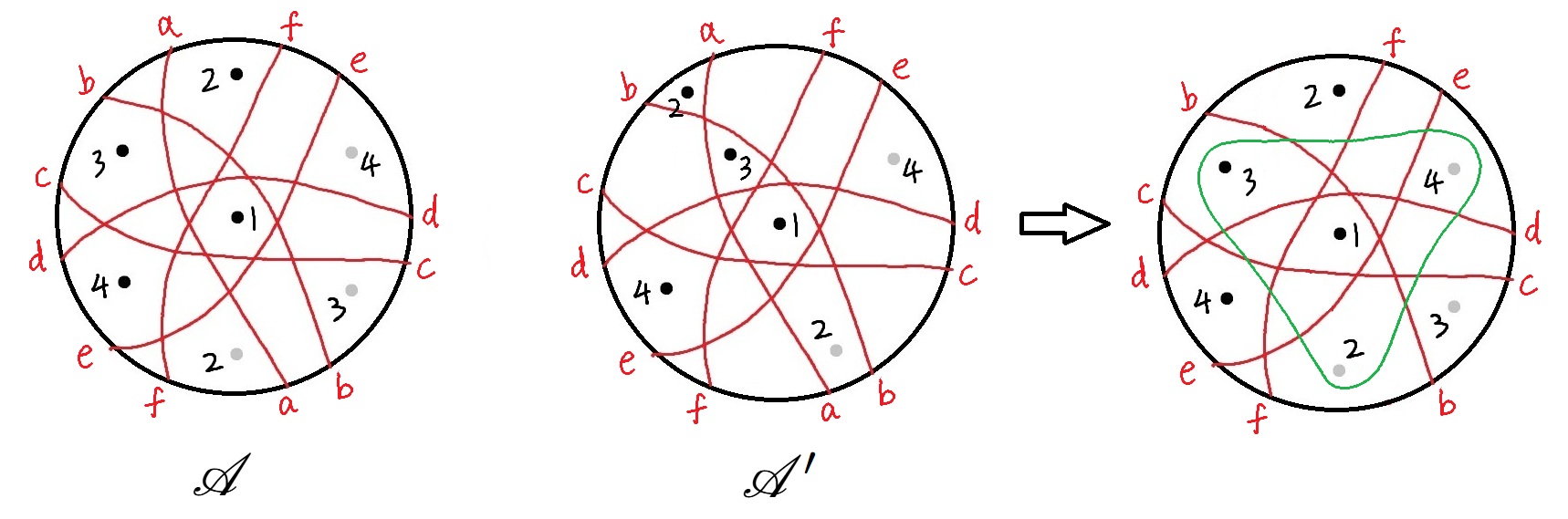}
    \caption{The loop arrangements $\mathscr{A}$ and $\mathscr{A}'$ are adjacent and differ by $l_a$.}
    \label{fig11}
    \end{figure}

   Let ${\mathcal{O}'}$ be the space arising from $\mathscr{A}'$ and $\overline{\mathcal{O}'}$ be its closure. We say that $\overline{\mathcal{O}}$ and $\overline{\mathcal{O}'}$ are \textbf{adjacent} if $\mathscr{A}$ and $\mathscr{A}'$ are adjacent loop arrangements. Suppose that $\mathscr{A}$ and $\mathscr{A}'$ differ only by the loop $l_i$. Then $\sigma_i$ is a common face of $\overline{\mathcal{O}}$ and $\overline{\mathcal{O}'}$. In this case, $\overline{\mathcal{O}}$ and $\overline{\mathcal{O}'}$ can be on the same side or on different sides of $\sigma_i$. We say that $\overline{\mathcal{O}}$ and $\overline{\mathcal{O}'}$ are \textbf{on the same side of $\mathbf{\sigma_i}$} if ${\mathcal{O}}$ and ${\mathcal{O}'}$ overlap in polyhedral surfaces sufficiently close to $\sigma_i$. Otherwise, we say that they are \textbf{on different sides of $\mathbf{\sigma_i}$}, or \textbf{$\mathbf{\overline{\mathcal{O}'}}$ is on the other side of $\mathbf{\sigma_i}$ with respect to $\mathbf{\overline{\mathcal{O}}}$}.

   Finally, each loop in $\mathscr{A}$ can be equipped with an \textbf{orientation}. Since a loop divides the sphere into two parts, it can have two possible orientations depending on which part is considered its inside or outside. We can choose one orientation from the two possibilities, and represent it by an arrow on that loop pointing outside. The orientation of a loop induces a direction on each edge transverse to that loop in $\mathscr{D}^*$. For example, consider the loop arrangement in Figure \ref{fig12}. The first picture shows the orientation we choose for the loop $l_a$, and the second picture shows the induced directions of some edges transverse to $l_a$ in $\mathscr{D}^*$. By choosing an orientation for each loop in $\mathscr{A}$, we can turn $\mathscr{D}^*$ into a directed graph.

    \begin{figure}[h]
    \centering    \includegraphics[width=0.7\textwidth]{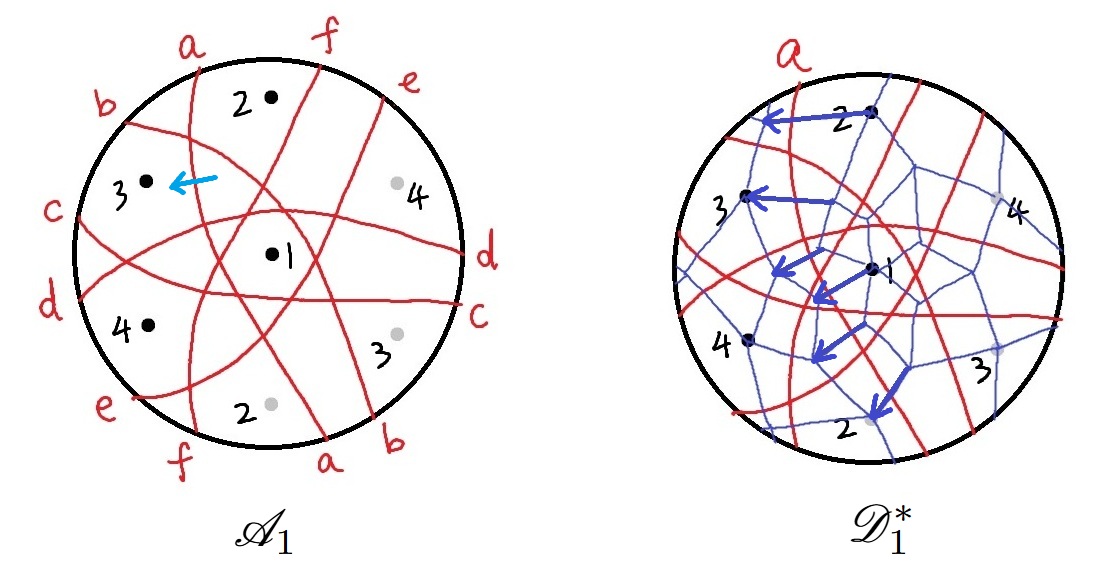}
    \caption{1) The orientation of $l_a$ is marked by a blue arrow; 2) The induced directed edges in $\mathscr{D}^*$.}
    \label{fig12}
    \end{figure}  


\section{Polyhedra with $8$ Vertices}

In this section, we focus on $\mathcal{C}_8(\delta_1, \delta_2, \delta_3, \delta_4)$, the space of centrally symmetric convex polyhedral surface with prescribed cone-deficits $\delta_1$, $\delta_2$, $\delta_3$, and $\delta_4$. The goal is to show that every polyhedral surface in this space has a parallelogram decomposition that is invariant under the antipodal map. Restricted to this section, a loop arrangement always contains $6$ loops labeled by $l_a$, $l_b$, $l_c$, $l_d$, $l_e$, and $l_f$ on a sphere with $8$ vertices labeled by $\pm 1$, $\pm 2$, $\pm 3$ and $\pm 4$.

 Let $\mathscr{A}_1$ be the loop arrangement in the first picture of Figure \ref{fig6}, and $\mathscr{D}_1^*$ be its dual graph. Let ${\mathcal{O}_1}$ the space arising from $\mathscr{A}_1$, and $\overline{\mathcal{O}_1}$ be its closure.

 Let $a$, $b$, $c$, $d$, $e$ and $f$ be positive variables. Recall from Section 3 that all quadrilaterals in $\mathscr{D}_1^*$ are considered parallelograms whose angles are determined by the prescribed cone-deficits. Thus, we have a map $F_1$ that sends every vector $(a,b,c,d,e,f)$ to a polyhedral surface $\mathscr{S}$ in $\mathcal{O}_1$ by assigning length $i$ to all edges in $\mathscr{D}_1^*$ transverse to $l_i$, where $i\in\{a,b,c,d,e,f\}$. We denote by $(0,\infty)^6$ the positive orthant of $\mathbb{R}^6$, which is the space of vectors in $\mathbb{R}^6$ where all components are positive. Similarly, denote by $[0,\infty)^6$ the non-negative orthant of $\mathbb{R}^6$. By definition, $F_1$ maps $(0,\infty)^6$ and $[0,\infty)^6$ onto $\mathcal{O}_1$ and $\overline{\mathcal{O}_1}$, respectively.

\begin{lemma} \label{lem1}
    The set $\mathcal{O}_1$ is open in  $\mathcal{C}_8(\delta_1, \delta_2, \delta_3, \delta_4)$.
\end{lemma}
\begin{proof}    

     It is clear that $F_1$ is continuous on $(0,\infty)^6$. We want to show that it is also injective on $(0,\infty)^6$.

    Let $\mathscr{S}$ be a polyhedral surface in $\mathcal{O}_1$ whose parallelogram decomposition constructed from $\mathscr{D}_1^*$ is sketched in the first picture of Figure \ref{fig13}. Suppose that $\mathscr{S}=F_1(a, b, c, d, e, f)$, where $a$, $b$, $\dots$, $f$ are the lengths of the edges transverse to $l_a$, $l_b$, $\dots$, $l_f$, respectively. Analogous to the proof of Theorem \ref{thm1}, we cut $\mathscr{S}$ open along the green edges (three of which are directed) in the first picture of Figure \ref{fig13} and unfold it to obtain a planar polygon $\mathcal{P}_{\mathscr{S}^+}$. This time we identify the plane with $\mathbb{R}^2$ instead of $\mathbb{C}$. We denote the three green directed edges by $2^+3^+$, $3^+4^+$ and $4^+2^-$. Their images in $\mathcal{P}_{\mathscr{S}^+}$ determine three vectors in $\mathbb{R}^2$, denoted by $Z_{2^+3^+}$, $Z_{3^+4^+}$ and $Z_{4^+2^-}$. From the first picture, we see that there exist unit vectors $\begin{pmatrix} a_x \\a_y\end{pmatrix}$, $\begin{pmatrix} b_x \\b_y\end{pmatrix}$, $\begin{pmatrix} c_x \\c_y\end{pmatrix}$, $\begin{pmatrix} d_x \\d_y\end{pmatrix}$, $\begin{pmatrix} e_x \\e_y\end{pmatrix}$ and $\begin{pmatrix} f_x \\f_y\end{pmatrix}$ in $\mathbb{R}^2$ such that 

    \begin{equation} \label{eq5}
        Z_{2^+3^+}=a\begin{pmatrix} a_x \\a_y\end{pmatrix}+b\begin{pmatrix} b_x \\b_y\end{pmatrix}, \quad Z_{3^+4^+}=c\begin{pmatrix} c_x \\c_y\end{pmatrix}+d\begin{pmatrix} d_x \\d_y\end{pmatrix}\quad \textrm{and} \quad Z_{4^+2^-}=e\begin{pmatrix} e_x \\e_y\end{pmatrix}+f\begin{pmatrix} f_x \\f_y\end{pmatrix}
    \end{equation}

    \begin{figure}[h]
    \centering    \includegraphics[width=0.95\textwidth]{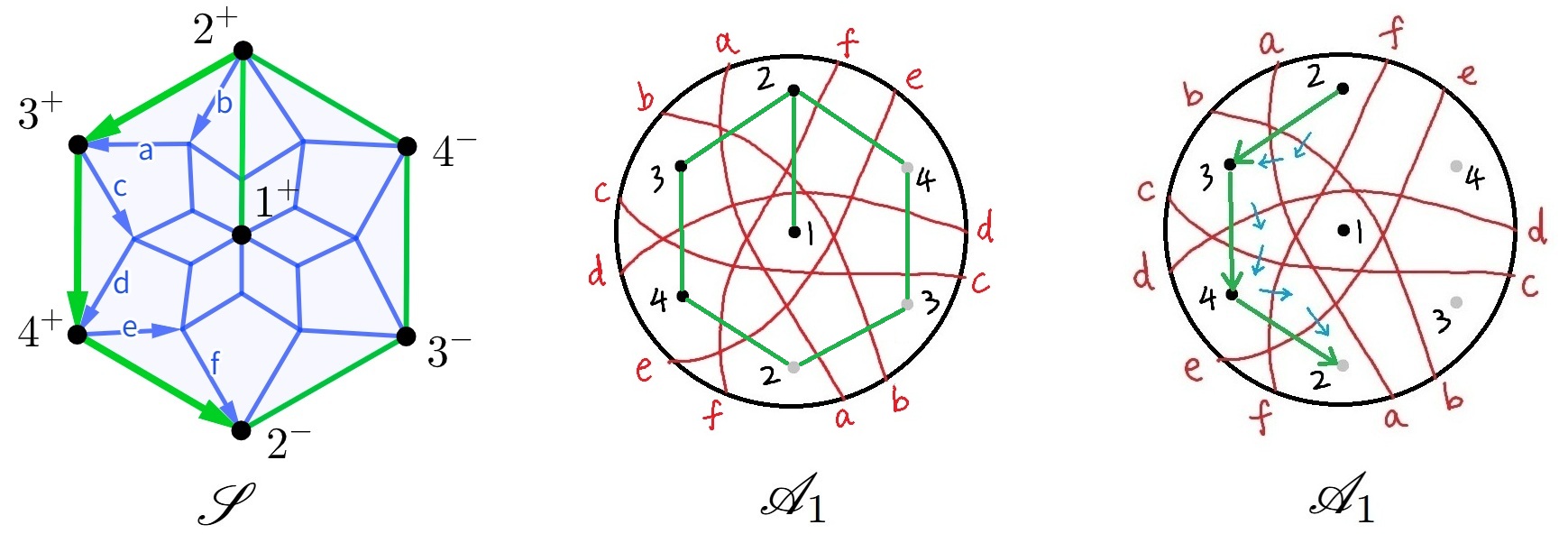}
    \caption{We can demonstrate a consistent way of unfolding polyhedral surfaces in $\mathcal{O}$ by drawing segments on a sphere with $\mathscr{A}_1$. }
    \label{fig13}
    \end{figure}

    Now we vary $\mathscr{S}$ in $\mathcal{O}_1$. The green edges in the first picture of Figure \ref{fig13} vary continuously with $\mathscr{S}$ without changing their homotopy classes. This allows us to define a consistent way of unfolding all surfaces in $\mathcal{O}_1$. Since such a way of unfolding exists throughout $\mathcal{O}_1$, we can demonstrate it on a picture of $\mathscr{A}_1$ rather than on a particular surface only. In the second picture in Figure \ref{fig13}, we use green edges that join the labeled vertices of the sphere to represent our consistent choices of the edges on the surfaces in $\mathcal{O}_1$ to cut along.
    
    In this way, we can define $\mathcal{P}_{\mathscr{S}^+}$ consistently (up to Euclidean isometries) for all surfaces in $\mathcal{O}_1$. Each $\mathcal{P}_{\mathscr{S}^+}$ gives rise to a vector $(Z_{2^+3^+}$, $Z_{3^+4^+}$, $Z_{4^+2^-})\in\mathbb{R}^6$. Up to rotations, we may assume that $Z_{2^+3^+}$, $Z_{3^+4^+}$, and $Z_{4^+2^-}$ can be expressed in terms of $a$, $b$, $\dots$, $f$ in Equation \ref{eq5}, where the unit vectors remain constant vectors. Let $\psi_1$ be the map on $(0,\infty)^6$ that sends every vector $(a,b,c,d,e,f)$ to $(Z_{2^+3^+}, Z_{3^+4^+}, Z_{4^+2^-})$ via unfolding surfaces in $\mathcal{O}_1$. Then we can write down the formula of $\psi_1$ as follows:    

    \begin{equation*} 
        \psi_1(a,b,c,d,e,f)=\begin{pmatrix} a_x & b_x & 0& 0 & 0 & 0\\ a_y & b_y & 0 & 0 & 0 & 0 \\ 0 & 0 & c_x & d_x & 0 & 0\\ 0 & 0 & c_y & d_y & 0 & 0 \\ 0 & 0 & 0 & 0 & e_x & f_x \\ 0 & 0 & 0 & 0 & e_y & f_y \end{pmatrix} \begin{pmatrix} a \\ b \\ c \\ d \\e \\ f   
    \end{pmatrix}
    \end{equation*}

    Denote the matrix above by $M_1$. Then 

    $$\text{det}(M_1)=\begin{vmatrix}a_x & b_x \\ a_y & b_y\end{vmatrix}\cdot\begin{vmatrix}c_x & d_x \\ c_y & d_y\end{vmatrix}\cdot\begin{vmatrix}e_x & f_x \\ e_y & f_y\end{vmatrix}=\sin(\frac{\delta_1}{2})\cdot\sin(\frac{\delta_1}{2})\cdot\sin(\frac{\delta_1}{2})>0$$
    where the determinant of each $2\times 2$ matrix is computed based on the angles in three parallelograms, which can be calculated by looking at the vertices enclosed in some lune bounded by two loops. One may refer to Equation \ref{eq1} in Section 3 for more details. 
    
    Therefore, $\psi_1$ is continuous and invertible. Since $\psi_1$ is a composition of $F_1$ by another map, $F_1$ is injective on $(0,\infty)^6$. By Theorem \ref{thm1}, $\mathcal{C}_8(\delta_1, \delta_2, \delta_3, \delta_4)$ has real dimension $6$. By the invariance of domain, $\mathcal{O}_1=F_1((0,\infty)^6)$ is open in $\mathcal{C}_8(\delta_1, \delta_2, \delta_3, \delta_4)$.     
\end{proof}

From the proof of Lemma \ref{lem1}, we see that $(Z_{2^+3^+}$, $Z_{3^+4^+}$, $Z_{4^+2^-})$ is a local frame defined on $\mathcal{O}_1$, so we can also visualize this local frame with $\mathscr{A}_1$. In the third picture of \ref{fig13}, we visualize its three components $Z_{2^+3^+}$, $Z_{3^+4^+}$, and $Z_{4^+2^-}$ as three green directed edges joining their corresponding labeled vertices on the sphere. This visualization is sometimes helpful to see the existence of the constant unit vectors $\begin{pmatrix} a_x \\a_y\end{pmatrix}$, $\begin{pmatrix} b_x \\b_y\end{pmatrix}$, $\dots$, $\begin{pmatrix} f_x \\f_y\end{pmatrix}$ in Equation \ref{eq5}. For example, let us orient the loops in $\mathscr{A}_1$ according to the blue arrows in the third picture, and consider the component $Z_{2^+3^+}$. The orientations of $l_a$ and $l_b$ induce directions on the edges transverse to them in $\mathscr{D}^*_1$, including the two edges labeled $a$ and $b$ in the first picture of Figure \ref{fig13}. We can then consider the triangle formed by these two edges and the edge $2^+3^+$, and see the existence of $\begin{pmatrix} a_x \\a_y\end{pmatrix}$ and $\begin{pmatrix} b_x \\b_y\end{pmatrix}$ to make Equation \ref{eq5} hold.

We say that the map $\psi_1$ in the proof of Lemma \ref{lem1} is \textbf{represented by the matrix $\bm{M_1}$}. It can be defined and is invertible on $[0,\infty)^6$. In addition, we know that $F_1$ maps $[0,\infty)^6$ onto $\overline{\mathcal{O}_1}$. This implies that $F_1$ is a homeomorphism from $[0,\infty)^6$ to $\overline{\mathcal{O}_1}$. Therefore, every element in $\overline{\mathcal{O}_1}$ corresponds to a unique vector in $[0,\infty)^6$ through $F_1$. This vector will be called its \textbf{coordinates with respect to $\mathbf{\mathscr{A}_1}$}, or simply \textbf{coordinates} when the underlying loop arrangement is clear. Similarly, we can also consider its $a$-coordinate, $b$-coordinate, etc.

The collection of elements in $\overline{\mathcal{O}_1}$ where at least two coordinates vanish is called its \textbf{codimension-two boundary}. For example, if an element in $\overline{\mathcal{O}_1}$ has zero $a$-coordinate and $b$-coordinate, while the remaining coordinates are positive, then it belongs to the codimension-two boundary. In this case, it is a surface in $\mathcal{C}_6(\delta_1, \delta_2+\delta_3, \delta_4)$. In fact, outside the codimension-two boundary, elements in $\overline{\mathcal{O}_1}$ belongs to $\mathcal{C}_8(\delta_1, \delta_2, \delta_3, \delta_4)$, so the local frame defined above is still valid in the interior of every face of $\overline{\mathcal{O}_1}$. Later, we will show that outside the codimension-two boundary, every element in $\overline{\mathcal{O}_1}$ has an open neighborhood in $\mathcal{C}_8(\delta_1, \delta_2, \delta_3, \delta_4)$.

According to the definition of $\overline{\mathcal{O}_1}$, every surface in $\overline{\mathcal{O}_1}$ has a parallelogram decomposition that is invariant under the antipodal map. Our goal of this section is to show that this result is true for all surfaces in $\mathcal{C}_8(\delta_1, \delta_2, \delta_3, \delta_4)$. Since there exist polyhedral surfaces in $\mathcal{C}_8(\delta_1, \delta_2, \delta_3, \delta_4)$ but outside $\overline{\mathcal{O}_1}$, we need to consider spaces arising from different loop arrangements.

Let $\mathscr{A}_2$ be the loop arrangement  $\mathscr{A}'$ in the second picture of Figure \ref{fig11}. According to the paragraph before Figure \ref{fig11}, $\mathscr{A}_1$ and $\mathscr{A}_2$ are adjacent loop arrangements since they differ by $l_a$ only. 

In Figure \ref{fig14}, we sketch $\mathscr{A}_2$ in the first picture. In fact, $\mathscr{A}_2$ is \textbf{combinatorically equivalent} to $\mathscr{A}_1$, which means that $\mathscr{A}_2$ becomes equivalent to $\mathscr{A}_1$ after we relabel its loops and vertices. One can see this by rotating the sphere so that $4^+$ is at the center of our view, as the second picture shows. Then one can check that each loop in the first picture can be moved to the loop with the same label in the second picture without changing its homotopy class. It may be helpful to consider which vertices are on the same side with respect to each loop. Then it is clear that $\mathscr{A}_2$ is combinatorically equivalent to $\mathscr{A}_1$ by comparing the second and the third picture.

    \begin{figure}[h]
    \centering    \includegraphics[width=0.95\textwidth]{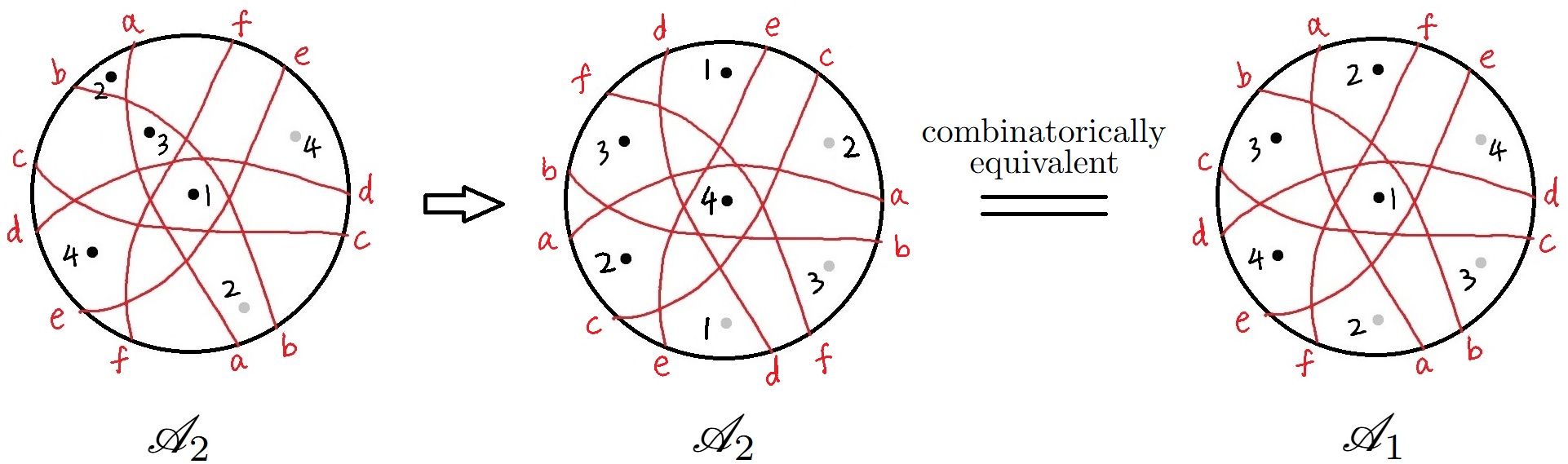}
    \caption{The loop arrangements $\mathscr{A}_1$ and $\mathscr{A}_2$ are combinatorically equivalent.}
    \label{fig14}
    \end{figure}

 Let ${\mathcal{O}_2}$ be the space arising from $\mathscr{A}_2$ and $\overline{\mathcal{O}_2}$ be its closure. Since $\mathscr{A}_2$ is combinatorically equivalent to $\mathscr{A}_1$, ${\mathcal{O}_2}$ is also open in $\mathcal{C}_8(\delta_1, \delta_2, \delta_3, \delta_4)$ by Lemma \ref{lem1} and symmetry. Since $\mathscr{A}_1$ and $\mathscr{A}_2$ differ by $l_a$ only, $\overline{\mathcal{O}_1}$ and $\overline{\mathcal{O}_2}$ are adjacent and coincide on the face $\sigma_a$. Then we have the following result: 

\begin{lemma} \label{lem2}
   The spaces $\overline{\mathcal{O}_1}$ and $\overline{\mathcal{O}_2}$ are on different sides of $\sigma_a$. 
\end{lemma}
\begin{proof}

    Let $(Z_{2^+3^+}$, $Z_{3^+4^+}$, $Z_{4^+2^-})$ be the local frame defined on $\mathcal{O}_1$ in the proof of Lemma \ref{lem1}. It is also defined in the interior of $\sigma_a$, since every element there still represents a surface in $\mathcal{C}_8(\delta_1, \delta_2, \delta_3, \delta_4)$.

    In Lemma \ref{lem1}, this local frame is obtained via unfolding every surface to a polygon $\mathcal{P}_{\mathscr{S}^+}$. When we think of this unfolding as the restriction of a developing map, every vertex of $\mathcal{P}_{\mathscr{S}^+}$ comes from a path on $\mathscr{S}$ we develop along from a base-point of $\mathscr{S}$ to a vertex of $\mathscr{S}$. When $\mathscr{S}$ moves across the interior of $\sigma_a$ into $\mathcal{O}_2$, we continuously vary this path with $\mathscr{S}$ without changing its homotopy class. This allows us to extend the definition of $(Z_{2^+3^+}$, $Z_{3^+4^+}$, $Z_{4^+2^-})$ to surfaces in $\mathcal{O}_2$, even though the actual cutting and unfolding may not always exist there. 

    In Figure \ref{fig15}, we visualize this local frame with $\mathscr{A}_1$ and $\mathscr{A}_2$. In $\mathscr{A}_1$, the orientations of all the loops are still the same as in Lemma \ref{lem1}. In $\mathscr{A}_2$, the orientations from $l_b$ to $l_f$ are the same as in $\mathscr{A}_1$. Finally, we orient $l_a$ in $\mathscr{A}_2$ according to the blue arrow marked on it.

    \begin{figure}[h]
    \centering    \includegraphics[width=0.7\textwidth]{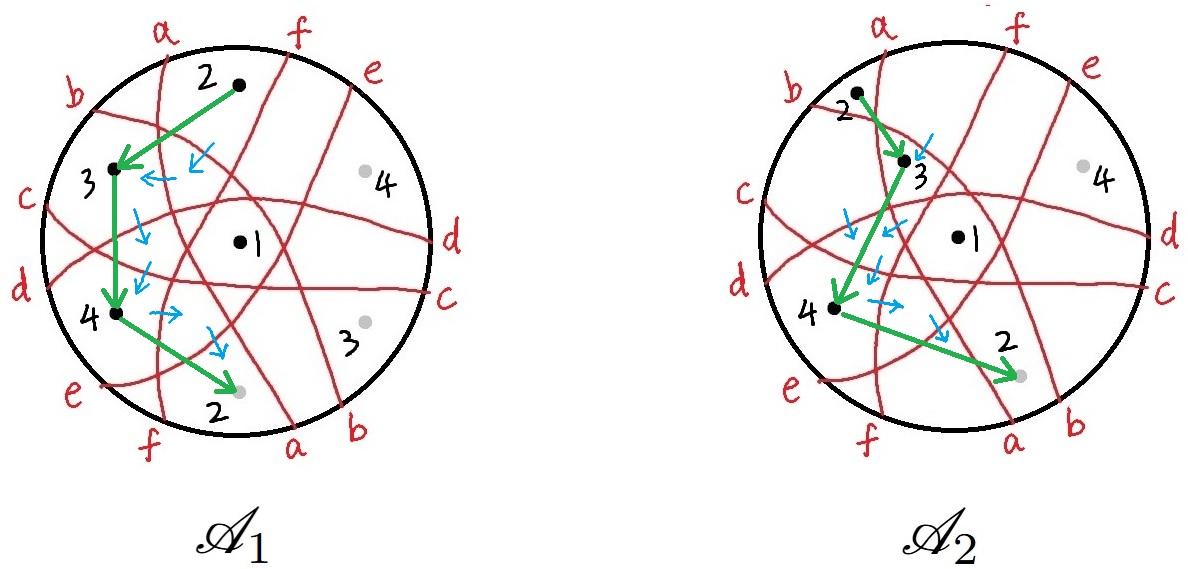}
    \caption{Visualizing the local frame $(Z_{2^+3^+}$, $Z_{3^+4^+}$, $Z_{4^+2^-})$ with $\mathscr{A}_1$ and $\mathscr{A}_2$.}
    \label{fig15}
    \end{figure}

    Let $\begin{pmatrix} b_x \\b_y\end{pmatrix}$, $\begin{pmatrix} c_x \\c_y\end{pmatrix}$, $\begin{pmatrix} d_x \\d_y\end{pmatrix}$, $\begin{pmatrix} e_x \\e_y\end{pmatrix}$, and $\begin{pmatrix} f_x \\f_y\end{pmatrix}$ be the unit vectors in Lemma \ref{lem1}. Then observe from Figure \ref{fig15} that there exists a constant unit vector $\begin{pmatrix} a'_x \\a'_y\end{pmatrix}$ such that we can write $Z_{2^+3^+}=-a\begin{pmatrix} a'_x \\a_y\end{pmatrix}+b\begin{pmatrix} b_x \\b_y\end{pmatrix}$, $Z_{3^+4^+}=a\begin{pmatrix} a'_x \\a_y\end{pmatrix}+c\begin{pmatrix} c_x \\c_y\end{pmatrix}+d\begin{pmatrix} d_x \\d_y\end{pmatrix}$ and $Z_{4^+2^-}=-a\begin{pmatrix} a'_x \\a_y\end{pmatrix}+e\begin{pmatrix} e_x \\e_y\end{pmatrix}+f\begin{pmatrix} f_x \\f_y\end{pmatrix}$ in terms of $a$, $b$, $\dots$, $f$. Here $\begin{pmatrix} a'_x\\ a'_y \end{pmatrix}$ is different from $\begin{pmatrix} a_x\\ a_y \end{pmatrix}$ since the loop $l_a$ has been changed.

    Thus, let $\psi_2$ be the map defined on $(0,\infty)^6$ that sends every vector $(a,b,c,d,e,f)$ to $(Z_{2^+3^+}$, $Z_{3^+4^+}$, $Z_{4^+2^-})$ via surfaces in $\mathcal{O}_2$. Then we have

    \begin{equation*}
        \psi_2(a,b,c,d,e,f)=\begin{pmatrix} -a'_x & b_x & 0& 0 & 0 & 0\\ -a'_y & b_y & 0 & 0 & 0 & 0 \\ a'_x & 0 & c_x & d_x & 0 & 0\\ a'_y & 0 & c_y & d_y & 0 & 0 \\ -a'_x & 0 & 0 & 0 & e_x & f_x \\ -a'_y & 0 & 0 & 0 & e_y & f_y \end{pmatrix} \begin{pmatrix} a \\ b \\ c \\ d \\e \\ f  
    \end{pmatrix}
    \end{equation*}  

    Denote the matrix in the above equation by $M_2$. Then

    $$\text{det}(M_2)=\begin{vmatrix}-a'_x & b_x \\ -a'_y & b_y\end{vmatrix}\cdot\begin{vmatrix}c_x & d_x \\ c_y & d_y\end{vmatrix}\cdot\begin{vmatrix}e_x & f_x \\ e_y & f_y\end{vmatrix}=-\sin({\frac{\delta_4}{2}})\sin^2({\frac{\delta_1}{2}})<0$$

    where the determinants of the three $2\times 2$ matrices can be calculated using Equation \ref{eq1}.

    Let $\psi_1$ and $M_1$ be defined as in Lemma \ref{lem1}. Suppose that there exists a surface $\mathscr{S}$ near $\sigma_a$ in ${\mathcal{O}_1}\cap{\mathcal{O}_2}$. This surface has (probably different) coordinates in $(0,\infty)^6$ with respect to both $\mathscr{A}_1$ and $\mathscr{A}_2$, which are mapped to the same image $(Z_{2^+3^+}$, $Z_{3^+4^+}$, $Z_{4^+2^-})$ by $\psi_1$ and $\psi_2$, respectively. We can then use Cramer's Rule to compute and compare its $a$-coordinates with respect to $\mathscr{A}_1$ and $\mathscr{A}_2$.
    
    Since $M_1$ and $M_2$ differ only by their first column, we obtain the same matrix by replacing their first columns with $(Z_{2^+3^+}$, $Z_{3^+4^+}$, $Z_{4^+2^-})$. However, the determinants of $M_1$ and $M_2$ have opposite signs. By Cramer's Rule, it is impossible that $\mathscr{S}$ has positive $a$-coordinates with respect to both $\mathscr{A}_1$ and $\mathscr{A}_2$, which contradicts the fact that $\mathscr{S}\in{\mathcal{O}_1}\cap{\mathcal{O}_2}$. Therefore, $\overline{\mathcal{O}_1}$ and $\overline{\mathcal{O}_2}$ are on different sides of $\sigma_a$.
\end{proof}

\begin{remark}
    It turns out that the proof of Lemma \ref{lem1} and Lemma \ref{lem2} works regardless of the local frame we choose. Suppose that we choose a different local frame. Then by definition of local frames, the change-of-coordinate map from $(Z_{2^+3^+}$, $Z_{3^+4^+}$, $Z_{4^+2^-})$ to the new local frame is given by an invertible matrix. Therefore, we just multiply this matrix to the left of $M_1$ and $M_2$. Then we still obtain two invertible matrices whose determinants have opposite signs.  
\end{remark}

Finally, we turn to the proof that every polyhedral surface in $\mathcal{C}_8(\delta_1, \delta_2, \delta_3, \delta_4)$ can be decomposed into parallelograms. It suffices for us to prove this result for $\overline{\mathcal{M}_8}(\delta_1, \delta_2, \delta_3, \delta_4)$, the metric completion of its moduli space.

Given a space $\mathcal{O}$ arising from a loop arrangement $\mathscr{A}$, we can consider the space of polyhedral surfaces in $\mathcal{O}$ with surface area $1$. We call this space the \textbf{moduli space arising from $\mathscr{A}$}, and denote it by $\Delta$. Let $\overline{\Delta}\subseteq \overline{\mathcal{M}_8}(\delta_1, \delta_2, \delta_3, \delta_4)$ be the metric completion of $\Delta$ with respect to the real hyperbolic metric in Theorem \ref{thm2}. We denote a face of $\overline{\Delta}$ by $\sigma_i({1})$ (where $1$ refers to the surface area) if it is a subset of the face $\sigma_i$ in $\overline{\mathcal{O}}$. If $\mathscr{A}'$ is another loop arrangement that differs from $\mathscr{A}$ by $l_i$ only, we say that $\overline{\Delta}$ and $\overline{\Delta'}$ are \textbf{on different sides of $\bm{\sigma_i({1})}$} if and only if $\overline{\mathcal{O}}$ and $\overline{\mathcal{O}'}$ are on different sides of $\sigma_i$.

Let $\overline{\Delta_1}$ be the metric completion of the moduli space arising from $\mathscr{A}_1$.

\begin{lemma} \label{lem3}
    The space $\overline{\Delta_1}$ has the structure of a real hyperbolic regular ideal $5$-simplex.
\end{lemma}
\begin{proof}
     Combinatorically, $\overline{\Delta_1}$ is a polytope with six vertices and six facets. Each vertex of $\overline{\Delta_1}$ has five vanishing coordinates among six with respect to $\mathscr{A}_1$. This corresponds to the case where the surface degenerates to a line. Note that each vertex of $\overline{\Delta_1}$ belongs to five facets of $\overline{\Delta_1}$, and each facet contains five vertices. Therefore, $\overline{\Delta_1}$ is a $5$-simplex by Proposition 2.16 in \cite{Zieg}. 
     
     Geometrically, we show that $\Delta_1$ is bounded by totally geodesic hyperplanes. Consider the set $F_1^{-1}(\Delta_1)$ in $(0,\infty)^6$. According to Theorem \ref{thm2}, by some linear transformation, $F_1^{-1}(\Delta_1)$ is mapped to a subset in the paraboloid $x_1^2-x_2^2-\dots-x_6^2=1$. This subset is bounded by the intersections of the paraboloid with six planes through the origin, which are the images of the coordinate planes under this linear transformation. Therefore, this subset and hence $\Delta_1$ is bounded by totally geodesic hyperplanes.

     Finally, $\overline{\Delta_1}$ is ideal because we have seen that its vertices do not belong to the real hyperbolic space $\mathbb{H}^5$. It is regular due to symmetry. This completes the proof of the result.
\end{proof}

By Lemma \ref{lem1} and Lemma \ref{lem2}, we also obtain the following result:

\begin{cor} \label{cor1}
    Every element in ${\Delta_1}$ or in the interior of a face of $\overline{\Delta_1}$ has an open neighborhood in $\mathcal{M}_8(\delta_1, \delta_2, \delta_3, \delta_4)$.
\end{cor}
\begin{proof}
    Let $x$ be an element in $\overline{\Delta_1}$. If $x\in \Delta_1$, the result follows from Lemma \ref{lem1}. If $x$ is in the interior of a face of $\overline{\Delta_1}$, without loss of generality, we assume that this face is $\sigma_a(1)$. 
    
    Let $\Delta_2$ be the moduli space arising from $\mathscr{A}_2$ and $\overline{\Delta_2}$ be its metric completion. We can choose a small $\epsilon$-neighborhood of $x$ in both $\overline{\Delta_1}$ and $\overline{\Delta_2}$. When $\epsilon$ is sufficiently small, the two neighborhoods are both open half-balls with radius $\epsilon$, and they have no intersection outside $\sigma_a(1)$ by Lemma \ref{lem2}. Therefore, their union is an open ball in $\mathcal{M}_8(\delta_1, \delta_2, \delta_3, \delta_4)$ containing $x$.
\end{proof}


Similar to $\mathscr{A}_2$, by permuting the the labels of vertices and loops in $\mathscr{A}_1$, we can obtain more loop arrangements that are different from but combinatorically equivalent to $\mathscr{A}_1$. It is clear that there are finitely many such loop arrangements. We consider the metric completions of the moduli spaces arising from these loop arrangements. For convenience, we will call them ``$5$-simplices'' due to Lemma \ref{lem3}. Later, we are going to construct infinitely many copies of such $5$-simplices and regard them as distinct spaces, though some of them arise from the same loop arrangement. Our goal is to glue these $5$-simplices to form a branched cover of $\overline{\mathcal{M}_8}(\delta_1, \delta_2, \delta_3, \delta_4)$.


Let $\overline{X}$ be a space constructed in the following way:
\begin{itemize}
    \item Initially, $\overline{X}$ is just a copy of the $5$-simplex $\overline{\Delta_1}$. 
    

    \item  Suppose that $\overline{X}$ has a boundary face that is the face  $\sigma_i(1)$ of some $5$-simplex $\overline{\Delta}$. By Lemma \ref{lem2} and symmetry, we can construct a copy of some $5$-simplex $\overline{\Delta'}$ so that $\overline{\Delta}$ and $\overline{\Delta'}$ are on different sides of $\sigma_i(1)$. Then we glue this copy of $\overline{\Delta'}$ to $\overline{X}$ along $\sigma_i(1)$.

     
    \item We repeat the process above so that every face in $\overline{X}$ is incident to exactly two $5$-simplices, one on each side.
\end{itemize}

Let $X_2$ be the union of the codimension-two boundaries of the $5$-simplices in $\overline{X}$. Since every face in $\overline{X}$ is incident to exactly two $5$-simplices, $\overline{X}\setminus X_2$ is a manifold. There is a continuous map $\iota: \overline{X}\rightarrow \overline{\mathcal{M}_8}(\delta_1, \delta_2, \delta_3, \delta_4)$ such that, restricted to every $5$-simplex in $\overline{X}$, $\iota$ is the inclusion map by viewing the $5$-simplex as a subset of $\overline{\mathcal{M}_8}(\delta_1, \delta_2, \delta_3, \delta_4)$. Although $\overline{X}$ contains infinitely many $5$-simplices, there are finitely many loop arrangements involved, so $\iota(\overline{X})$ is a finite union of closed sets. Hence, $\iota(\overline{X})$ is closed in $\overline{\mathcal{M}_8}(\delta_1, \delta_2, \delta_3, \delta_4)$, which implies that $\iota(\overline{X})\setminus\iota(X_2)$ is closed in $\overline{\mathcal{M}_8}(\delta_1, \delta_2, \delta_3, \delta_4)\setminus\iota(X_2)$.

By Corollary \ref{cor1}, every surface in $\iota(\overline{X})\setminus\iota(X_2)$ has a neighborhood contained in $\overline{\mathcal{M}_8}(\delta_1, \delta_2, \delta_3, \delta_4)$. Since $\iota(X_2)$ has codimension two, the neighborhood is away from $\iota(X_2)$ if it is sufficiently small. Therefore, $\iota(\overline{X})\setminus\iota(X_2)$ is also open in $\overline{\mathcal{M}_8}(\delta_1, \delta_2, \delta_3, \delta_4)\setminus\iota(X_2)$. Finally, $\overline{\mathcal{M}_8}(\delta_1, \delta_2, \delta_3, \delta_4)\setminus\iota(X_2)$ is non-empty due to their different dimensions, so we have shown that $\iota(\overline{X})\setminus\iota(X_2)=\overline{\mathcal{M}_8}(\delta_1, \delta_2, \delta_3, \delta_4)\setminus\iota(X_2)$.

Thus, every surface in $\overline{\mathcal{M}_8}(\delta_1, \delta_2, \delta_3, \delta_4)\setminus\iota(X_2)$ has a parallelogram decomposition. By definition of $X_2$, every surface in $\iota(X_2)$ has a parallelogram decomposition as well. The decomposition is invariant under the antipodal map by our construction. In addition, there are $6$ loops in each loop arrangement, every two of which meet at two antipodal points since they are great circles. In conclusion, we have proved the following result:

\begin{theorem} \label{thm4}
    Every surface in $\overline{\mathcal{M}_8}(\delta_1, \delta_2, \delta_3, \delta_4)$ can be decomposed into at most $2\binom{6}{2}=30$ parallelograms, and the decomposition is invariant under the antipodal map.
\end{theorem}

In summary, we have established two methods to build coordinate charts on the space of polyhedral surfaces. One is through unfolding the surfaces, and the other is through decomposing them into parallelograms. Although the first method applies to more general convex polyhedral surfaces, the second method is also interesting because it sometimes provides a more concrete geometric description of moduli spaces. 

For instance, consider the space of centrally symmetric polyhedral surfaces with $8$ unlabeled vertices, whose cone-deficits all equal to $\frac{\pi}{2}$. In this unlabeled space, two surfaces are equivalent if they differ by a Euclidean isometry. Denote its unlabeled moduli space by $\mathcal{M}_8(\frac{\pi}{2})$, and its metric completion by $\overline{\mathcal{M}_8}(\frac{\pi}{2})$. The following result is a direct consequence of Lemma \ref{lem3}:

\begin{cor} \label{cor2}
    The space $\overline{\mathcal{M}_8}(\frac{\pi}{2})$ has the structure of the quotient space $\overline{\Delta_1}/\Gamma$, where $\overline{\Delta_1}$ is defined as before, and $\Gamma$ is isomorphic to the dihedral group $D_6$ with order $12$.
\end{cor}
\begin{proof}
    We can interpret $\overline{\mathcal{M}_8}(\frac{\pi}{2})$ as the space of equivalence classes in $\overline{\Delta_1}$. The main task is to figure out the equivalence relation. 

    Figure \ref{fig16} shows the front view of a parallelogram decomposition of a surface in $\Delta_1$, whose coordinates are $(a,b,c,d,e,f)$ with respect to $\mathscr{A}_1$. Some edge-lengths of the parallelograms are labeled. One can think of the coordinates as reading the edge-lengths consecutively in the red polygonal path in Figure \ref{fig16} whose direction is indicated by the red arrow. 

    \begin{figure}[h]
    \centering    \includegraphics[width=0.3\textwidth]{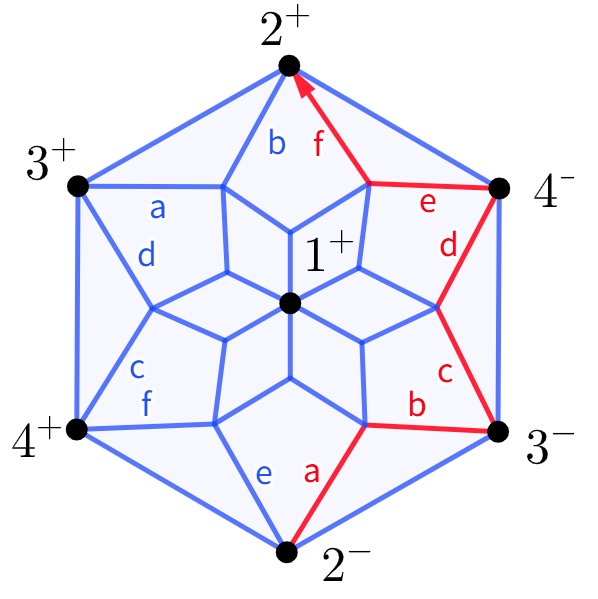}
    \caption{Parallelogram decomposition of a surface in $\Delta_1$ with coordinates $(a,b,c,d,e,f)$.}
    \label{fig16}
    \end{figure}

    In the unlabeled moduli space $\overline{\mathcal{M}_8}(\frac{\pi}{2})$, we don't distinguish vertices of the surfaces. Thus, if we rotate the parallelogram decomposition so that the red polygonal path starts at a different vertex, we get an equivalent surface in $\overline{\mathcal{M}_8}(\frac{\pi}{2})$ despite that their coordinates with respect to $\mathscr{A}_1$ are different. For example, $(a,b,c,d,e,f)$ and $(c,d,e,f,b,a)$ correspond to an equivalent surface in $\overline{\mathcal{M}_8}(\frac{\pi}{2})$. In addition, we can reflect the parallelogram decomposition so that the coordinates are read in reverse order. For example, $(a,b,c,d,e,f)$ and $(f,e,d,c,b,a)$ correspond to an equivalent surface in $\overline{\mathcal{M}_8}(\frac{\pi}{2})$. Thus, the group $\Gamma$ is isomorphic to the symmetry group of a regular hexagon, which is the dihedral group $D_6$.
\end{proof}

\section{Polyhedra with $10$ Vertices}

In this section, we study the space of centrally symmetric convex polyhedral surfaces with $10$ vertices with prescribed cone-deficits $\delta_1$, $\delta_2$, $\delta_3$, $\delta_4$ and $\delta_5$. Analogous to the last section, the goal of this section is to show that every surface in $\mathcal{C}_{10}(\delta_1, \delta_2, \delta_3, \delta_4, \delta_5)$ can be decomposed into parallelograms. Restricted to this section, a loop arrangement always has $8$ loops labeled by $l_a$, $l_b$, $l_c$, $\dots$, $l_g$, $l_h$ on a sphere with $8$ vertices labeled by $\pm 1$, $\pm 2$, $\dots$, $\pm 5$.

The main idea is still to construct spaces arising from different loop arrangements. This time, we consider the four loop arrangements $\mathscr{A}_1$, $\mathscr{A}_2$, $\mathscr{A}_3$, and $\mathscr{A}_4$ in Figure \ref{fig17}. Let ${\mathcal{O}_1}$, ${\mathcal{O}_2}$, ${\mathcal{O}_3}$, and ${\mathcal{O}_4}$ be the spaces arising from $\mathscr{A}_1$, $\mathscr{A}_2$, $\mathscr{A}_3$, and $\mathscr{A}_4$, respectively. In addition, we say that a space ${\mathcal{O}}$ has \textbf{Type 1} if it arises from a loop arrangement combinatorically equivalent to $\mathscr{A}_1$ up to relabeling vertices or loops. Similarly, we define spaces of \textbf{Type 2}, \textbf{Type 3}, and \textbf{Type 4}. Theorem \ref{thm3} implies that these spaces are subsets of $\mathcal{C}_{10}(\delta_1, \delta_2, \delta_3, \delta_4, \delta_5)$.

\begin{figure}[h]
    \centering    \includegraphics[width=0.95\textwidth]{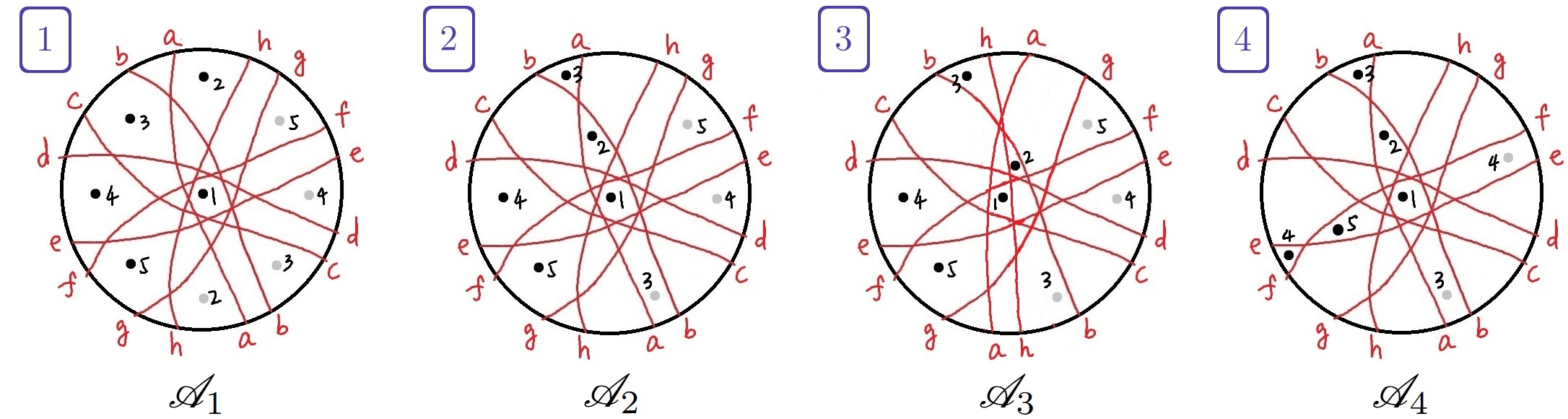}
    \caption{Loop arrangements $\mathscr{A}_1$, $\mathscr{A}_2$, $\mathscr{A}_3$, and $\mathscr{A}_4$. The numbers in boxes refer to the types of spaces arising from these loop arrangements.}
    \label{fig17}
\end{figure}

\begin{lemma} \label{lem4}
    Every Type 1 space is open in $\mathcal{C}_{10}(\delta_1, \delta_2, \delta_3, \delta_4, \delta_5)$.
\end{lemma}

\begin{proof}
    It suffices to prove the result for ${\mathcal{O}_1}$. The proof is similar to that of Lemma \ref{lem1}.

    In the first picture of Figure \ref{fig18}, we demonstrate a consistent way of unfolding surfaces in ${\mathcal{O}_1}$ by cutting along the edges of these surfaces represented by the green segments joining the labeled vertices on the sphere. This gives rise to a local frame $(Z_{2^+3^+}$, $Z_{3^+4^+}$, $Z_{4^+5^+}$, $Z_{5^+2^-})$ on $\mathcal{O}_1$, which is visualized in the second picture. We orient the loops in $\mathscr{A}_1$ according to the blue arrows. Then we can observe from Figure \ref{fig18} that there exist constant unit vectors $\begin{pmatrix}a_x \\a_y\end{pmatrix}$, $\begin{pmatrix}b_x \\b_y\end{pmatrix}$, $\dots$, $\begin{pmatrix}h_x \\h_y\end{pmatrix}$ in $\mathbb{R}^2$, using which we can express $Z_{2^+3^+}$, $Z_{3^+4^+}$, $Z_{4^+5^+}$ and $Z_{5^+2^-}$ in terms of $a$, $b$, $\dots$, $h$ as we did in Lemma \ref{lem1}. 

    \begin{figure}[h]
    \centering    \includegraphics[width=0.7\textwidth]{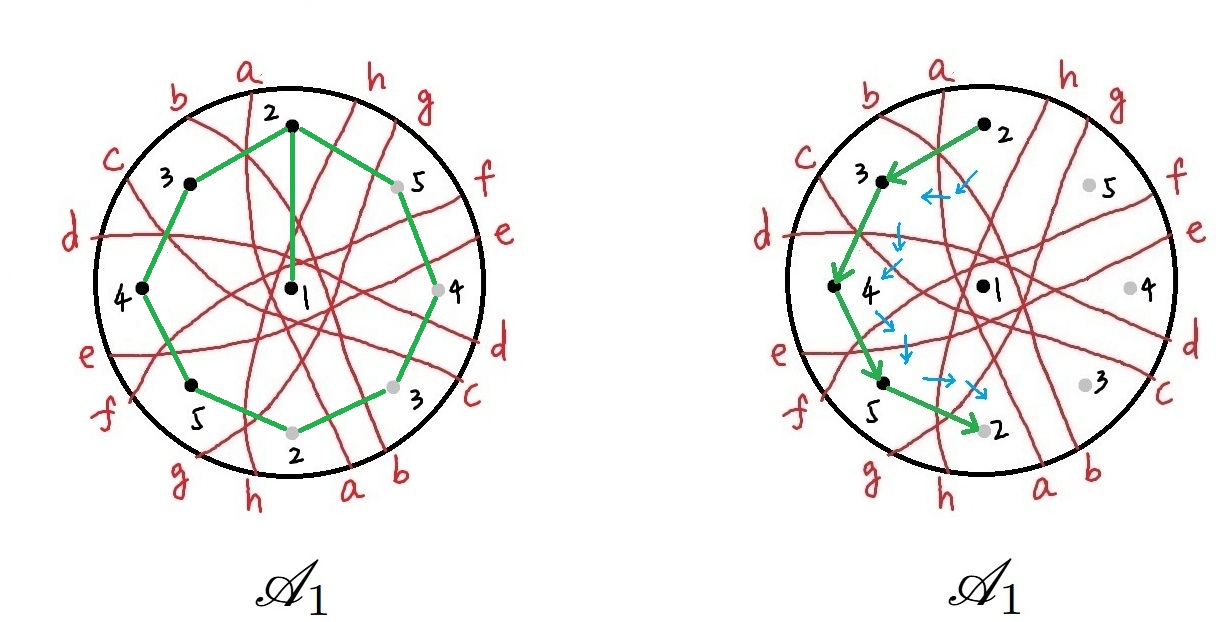}
    \caption{Visualization of the unfolding and the local frame it gives rise to with $\mathscr{A}_1$.}
    \label{fig18}
    \end{figure}
    
    Similar to the constructions in Section 4, let $F_1$ be the map that sends every vector $(a,b,\dots,g,h)\in(0,\infty)^8$ to a surface in $\mathcal{O}_1$. Let $\psi_1$ be the map that sends every vector $(a,b,\dots,g,h)\in(0,\infty)^8$ to $(Z_{2^+3^+}$, $Z_{3^+4^+}$, $Z_{4^+5^+}$, $Z_{5^+2^-})$ via unfolding surfaces in $\mathcal{O}_1$. Then we have  
    
    $$\psi_1(a,b,c,d,e,f,g,h)=\begin{pmatrix} a_x & b_x & 0& 0 & 0 & 0 & 0 & 0\\ a_y & b_y & 0 & 0 & 0 & 0 & 0 & 0 \\ 0 & 0 & c_x & d_x & 0 & 0 & 0 & 0\\ 0 & 0 & c_y & d_y & 0 & 0 & 0 & 0 \\ 0 & 0 & 0 & 0 & e_x & f_x & 0 & 0 \\ 0 & 0 & 0 & 0 & e_y & f_y & 0 & 0 \\ 0 & 0 & 0 & 0 & 0 & 0 & g_x & h_x \\ 0 & 0 & 0 & 0 & 0 & 0 & g_y & h_y\end{pmatrix} \begin{pmatrix} a \\ b \\ c \\ d \\e \\ f \\g \\h  
    \end{pmatrix}$$

    Denote the matrix in the above equation by $M_1$. Then 

    $$\text{det}(M_1)=\begin{vmatrix}a_x & b_x \\ a_y & b_y\end{vmatrix}\cdot\begin{vmatrix}c_x & d_x \\ c_y & d_y\end{vmatrix}\cdot\begin{vmatrix}e_x & f_x \\ e_y & f_y\end{vmatrix}\cdot\begin{vmatrix}g_x & h_x \\ g_y & h_y\end{vmatrix}=\sin^4({\frac{\delta_1}{2}})>0$$

    Thus, $\psi_1$ is continuous and invertible. Since $\psi_1$ is the composition of $F_1$ with another map, $F_1$ is continuous and injective on $(0,\infty)^8$. As $\mathcal{O}_1=F_1((0,\infty)^8)$ and ${\mathcal{C}_{10}}(\delta_1, \delta_2. \delta_3, \delta_4, \delta_5)$ has dimension $8$, the result follows from the invariance of domain. 
\end{proof}

Let $\overline{\mathcal{O}_1}$ be the closure of $\mathcal{O}_1$. Analogous to Section 4, we define the coordinates with respect to $\mathscr{A}_1$ and then the codimension-two boundary of $\overline{\mathcal{O}_1}$. Away from the codimension-two boundary of $\overline{\mathcal{O}_1}$, every element in $\overline{\mathcal{O}_1}$ represents a surface in ${\mathcal{C}_{10}}(\delta_1, \delta_2. \delta_3, \delta_4, \delta_5)$, and the local frame $(Z_{2^+3^+}$, $Z_{3^+4^+}$, $Z_{4^+5^+}$, $Z_{5^+2^-})$ in Lemma \ref{lem4} is defined.  

Next, we will establish two technical results that generalize Lemma \ref{lem2}. 

Let $\mathscr{A}$ and $\mathscr{A}'$ be two adjacent loop arrangements that differ by the loop $l_i$. Let ${\mathcal{O}}$ and ${\mathcal{O}'}$ be the spaces arising from $\mathscr{A}$ and $\mathscr{A}'$, respectively. Let $\overline{\mathcal{O}}$ and $\overline{\mathcal{O}'}$ be their closures. Suppose that there is a local frame defined on ${\mathcal{O}}$. Since $\overline{\mathcal{O}}$ and $\overline{\mathcal{O}'}$ are adjacent and meet on the face $\sigma_i$, the local frame on ${\mathcal{O}}$ can be extended continuously to ${\mathcal{O}'}$ through the interior of $\sigma_i$ in a way we described in Lemma \ref{lem2}. Let $\psi$ and $\psi'$ be the maps from $(0,\infty)^8$ to this local frame via surfaces in $\mathcal{O}$ and $\mathcal{O}'$, respectively. Let $M$ and $M'$ be the matrices representing $\psi$ and $\psi'$, respectively.

Let $Z_{TU}$ and $Z_{VW}$ be two components in the local frame above, where $T$, $U$, $V$ and $W$ are four labeled vertices. The first technical result involves $Z_{TU}$ only.

\begin{lemma} \label{lem5}
    Suppose that there exist constant unit vectors $\begin{pmatrix} i_x\\ i_y \end{pmatrix}$, $\begin{pmatrix} i'_x\\ i'_y \end{pmatrix}$ and $\begin{pmatrix} j_x\\ j_y \end{pmatrix}$ such that $Z_{TU}=i\begin{pmatrix} i_x\\ i_y \end{pmatrix}+j\begin{pmatrix} j_x\\ j_y \end{pmatrix}$ in the formula of $\psi$, and $Z_{TU}=i\begin{pmatrix} i'_x\\ i'_y \end{pmatrix}+j\begin{pmatrix} j_x\\ j_y \end{pmatrix}$ in the formula of $\psi'$. Then $\overline{\mathcal{O}}$ and $\overline{\mathcal{O}'}$ are on different sides of $\sigma_i$ if and only if the determinants $\begin{vmatrix}i_x & j_x \\ i_y & j_y\end{vmatrix}$ and $\begin{vmatrix}i'_x & j_x \\ i'_y & j_y\end{vmatrix}$ have opposite signs. 
\end{lemma}

\begin{proof}
    The main idea of proof is to show that $M$ and $M'$ have opposite signs in their determinants.
    
    Without loss of generality, we may assume that $Z_{TU}$ is the first component of the local frame, and $i$, $j$ are the first two components of the vectors in $(0,\infty)^8$. Then the first two rows in $M$ are $\begin{pmatrix}i_x & j_x & 0 & \dots & 0 \\ i_y & j_y & 0 & \dots & 0\end{pmatrix}$. Similarly, the first two rows in $M'$ are $\begin{pmatrix}i'_x & j_x & 0 & \dots & 0 \\ i'_y & j_y & 0 & \dots & 0\end{pmatrix}$. 
    
    Note that $M$ and $M'$ differ only by the first column. Therefore, by deleting the first two rows and the first two columns from $M$ and $M'$, respectively, we get the same sub-matrix. By multiplying the determinant of this sub-matrix to $\begin{vmatrix}i_x & j_x \\ i_y & j_y\end{vmatrix}$ and $\begin{vmatrix}i'_x & j_x \\ i'_y & j_y\end{vmatrix}$, respectively, we get the determinants of $M$ and $M'$. Thus, $M$ and $M'$ have opposite signs if and only if $\begin{vmatrix}i_x & j_x \\ i_y & j_y\end{vmatrix}$ and $\begin{vmatrix}i'_x & j_x \\ i'_y & j_y\end{vmatrix}$ have opposite signs. The rest of the proof is same as that of Lemma \ref{lem2} in principle.
\end{proof}

To apply lemma \ref{lem5} in practice, we can visualize the component $Z_{TU}$ of the local frame with both $\mathscr{A}$ and $\mathscr{A}'$. There is a natural way to orient the loops $l_i$ and $l_j$ in both $\mathscr{A}$ and $\mathscr{A}'$ according to the formula of $Z_{TU}$ in Lemma \ref{lem5}. It will be helpful to look at the arrows on these loops that indicate their orientations. 

For example, consider $l_i$ and $l_j$ in $\mathscr{A}$ and $\mathscr{A}$ in the first picture of Figure \ref{fig19}. Based on the formula of $Z_{TU}$ in Lemma \ref{lem5}, the loops $l_i$ and $l_j$ should be oriented by the blue arrows. In $\mathscr{A}$, the arrow on $l_j$ overlaps with the arrow on $l_i$ after a counterclockwise rotation by an angle $<\pi$. This means $\begin{vmatrix}i_x & j_x \\ i_y & j_y\end{vmatrix}>0$. In $\mathscr{A}'$, the arrow on $l_j$ overlaps with the arrow on $l_i$ after a clockwise rotation by an angle $<\pi$, which means $\begin{vmatrix}i_x & j_x \\ i_y & j_y\end{vmatrix}<0$. By Lemma \ref{lem5}, $\overline{\mathcal{O}}$ and $\overline{\mathcal{O}'}$ are on different sides of $\sigma_i$. Similarly, the reader may check that for $\mathscr{A}$ and $\mathscr{A}'$ in the second picture, $\overline{\mathcal{O}}$ and $\overline{\mathcal{O}'}$ are on the same side of $\sigma_i$.

    \begin{figure}[h]
    \centering    \includegraphics[width=0.95\textwidth]{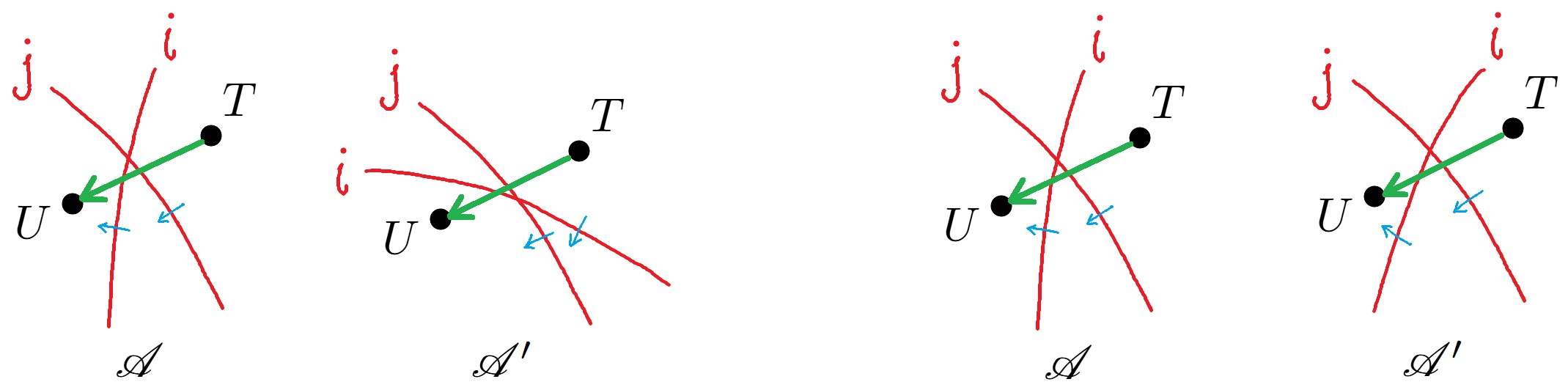}
    \caption{Applying Lemma \ref{lem5} by checking the orientations of $l_i$ and $l_j$ in $\mathscr{A}$ and $\mathscr{A}'$}
    \label{fig19}
    \end{figure}

The second technical result involves both $Z_{TU}$ and $Z_{VW}$.

\begin{lemma} \label{lem6}
    Suppose there exist constant unit vectors $\begin{pmatrix} i_x \\i_y \end{pmatrix}$, $\begin{pmatrix} i'_x \\i'_y \end{pmatrix}$, $\begin{pmatrix} j_x \\j_y \end{pmatrix}$, $\begin{pmatrix} k_x \\k_y \end{pmatrix}$ and $\begin{pmatrix} m_x \\m_y \end{pmatrix}$ such that $Z_{TU}=i\begin{pmatrix} i_x\\ i_y \end{pmatrix}+j\begin{pmatrix} j_x\\ j_y \end{pmatrix}+m\begin{pmatrix} m_x\\ m_y \end{pmatrix}$ and $Z_{VW}=k\begin{pmatrix} k_x\\ k_y \end{pmatrix}-m\begin{pmatrix} m_x\\ m_y \end{pmatrix}$ in the formula of $\psi$. In addition, $Z_{TU}=j\begin{pmatrix} j_x\\ j_y \end{pmatrix}+m\begin{pmatrix} m_x\\ m_y \end{pmatrix}$ and $Z_{VW}=i\begin{pmatrix} i'_x\\ i'_y \end{pmatrix}+k\begin{pmatrix} k_x\\ k_y \end{pmatrix}-m\begin{pmatrix} m_x\\ m_y \end{pmatrix}$ in the formula of $\psi'$. Then $\overline{\mathcal{O}}$ and $\overline{\mathcal{O}'}$ are on different sides of the face $\sigma_i$ if and only if the products of the determinants $\begin{vmatrix}i_x & j_x \\ i_y & j_y\end{vmatrix}\cdot \begin{vmatrix}k_x & m_x \\ k_y & m_y\end{vmatrix}$ and $\begin{vmatrix}i'_x & k_x \\ i'_y & k_y\end{vmatrix}\cdot\begin{vmatrix}j_x & m_x \\ j_y & m_y\end{vmatrix}$ have opposite signs.
\end{lemma}

\begin{proof}
    Again, we need to show that the determinants of $M$ and $M'$ have opposite signs. Without loss of generality, we assume that $Z_{TU}$ and $Z_{VW}$ are the first two components of the local frame, and $i$, $j$, $k$, $m$ are the first four components of the vectors in $[0,\infty)^8$. Then the first four rows of $M$ and $M'$ are $\begin{pmatrix}
        i_x & j_x & 0 & m_x & 0 & \dots & 0 \\
        i_y & j_y & 0 & m_y & 0 & \dots & 0 \\
        0 & 0 & k_x & -m_x & 0 & \dots & 0 \\
        0 & 0 & k_y & -m_y & 0 & \dots & 0 \\
    \end{pmatrix}$ and $\begin{pmatrix}
        0 & j_x & 0 & m_x & 0 & \dots & 0 \\
        0 & j_y & 0 & m_y & 0 & \dots & 0 \\
        i'_x & 0 & k_x & -m_x & 0 & \dots & 0 \\
        i'_y & 0 & k_y & -m_y & 0 & \dots & 0 \\
    \end{pmatrix}$, respectively.

    Since $M$ and $M'$ differ only by the first column, their sub-matrices after deleting the first four rows and columns are equivalent. Thus, to compare the signs of their determinants, it suffices to compare the signs of their $4\times4$-minor on the upper-left corner. 
    
    For $M$, this minor is $\begin{vmatrix}i_x & j_x & 0 & m_x \\ i_y & j_y & 0 & m_x \\0 & 0 & k_x & -m_x \\ 0 & 0 & k_y & -m_y \end{vmatrix}=-\begin{vmatrix}i_x & j_x \\ i_y & j_y\end{vmatrix}\cdot \begin{vmatrix}k_x & m_x \\ k_y & m_y\end{vmatrix}$. For $M'$, it is $\begin{vmatrix} 0 & j_x & 0 & m_x \\ 0 & j_y & 0 & m_y \\i'_x & 0 & k_x & -m_x \\ i'_y & 0 & k_y & -m_y \end{vmatrix}=-\begin{vmatrix}i'_x & k_x \\ i'_y & k_y\end{vmatrix}\cdot\begin{vmatrix}j_x & m_x \\ j_y & m_y\end{vmatrix}$. Thus, the determinants of $M$ and $M'$ have opposite signs if and only if these two products have opposite signs. The rest of the proof is similar to that of Lemma \ref{lem2}.
\end{proof}

Again, to apply Lemma \ref{lem6} in practice, it will be helpful to visualize $Z_{TU}$ and $Z_{VW}$ with $\mathscr{A}$ and $\mathscr{A}'$. There is a natural way to orient the loops $l_i$, $l_j$, $l_k$ and $l_m$ in $\mathscr{A}$ and $\mathscr{A}'$ according to the formulas of $Z_{TU}$ and $Z_{VW}$ in Lemma \ref{lem6}. The sign of the products in Lemma \ref{lem6} can be computed more quickly by looking at the arrows that indicate their orientations.

\begin{figure}[h]
    \centering
    \includegraphics[width=0.7\linewidth]{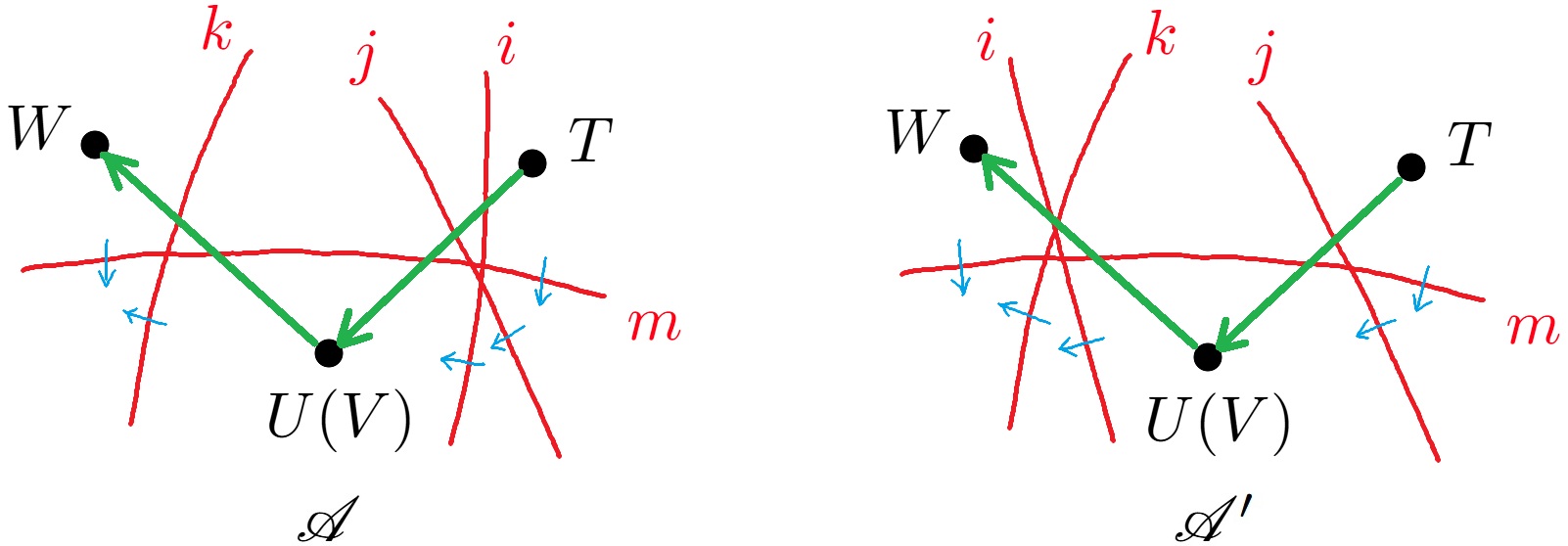}
    \caption{Applying Lemma \ref{lem6} by checking the orientations of $l_i$, $l_j$, $l_k$ and $l_m$}
    \label{fig20}
\end{figure}

For example, Figure \ref{fig20} shows a particular case where $U=V$. Consider the loops $l_i$, $l_j$, $l_k$ and $l_m$ in $\mathscr{A}$ and $\mathscr{A}'$. Their orientations based on the formulas of $Z_{TU}$ and $Z_{VW}$ are marked by blue arrows. In $\mathscr{A}$, the blue arrow on $l_i$ overlaps with that on $l_j$ after a clockwise rotation by an angle $<\pi$, which means $\begin{vmatrix}i_x & j_x \\ i_y & j_y\end{vmatrix}>0$. Similarly, we have $\begin{vmatrix}k_x & m_x \\ k_y & m_y\end{vmatrix}>0$. In $\mathscr{A}'$, one can check from the blue arrows that $\begin{vmatrix}i'_x & k_x \\ i'_y & k_y\end{vmatrix}<0$ and $\begin{vmatrix}j_x & m_x \\ j_y & m_y\end{vmatrix}>0$. Therefore, $\overline{\mathcal{O}}$ and $\overline{\mathcal{O}'}$ are on different sides of $\sigma_i$ by Lemma \ref{lem6}.

Now let $\mathcal{O}$ be a space of Type 1, 2, 3 or 4 and $\overline{\mathcal{O}}$ be its closure. For every face $\sigma$ of $\overline{\mathcal{O}}$, we will apply Lemma \ref{lem5} and Lemma \ref{lem6} to show that there exists a space $\mathcal{O}'$ of Type 1, 2, 3 or 4 whose closure  $\overline{\mathcal{O}'}$ is on the other side of $\sigma$ with respect to $\overline{\mathcal{O}}$.

\subsection*{Type 1 Spaces}

Without loss of generality, we can assume that $\mathcal{O}$ is ${\mathcal{O}_1}$. Due to the symmetry of $\mathscr{A}_1$, we can also assume that the face $\sigma$ is $\sigma_b$. 

\begin{lemma} \label{lem7}
    The closures $\overline{\mathcal{O}_1}$ and $\overline{\mathcal{O}_2}$ are on different sides of $\sigma_b$.
\end{lemma}
\begin{proof}  

    Note that $\mathscr{A}_2$ is obtained from $\mathscr{A}_1$ by moving the labeled vertices $2^+$, $2^-$, $3^+$, and $3^-$ in $\mathscr{A}_1$ across the loop $l_b$ simultaneously. Thus. similar to $\mathscr{A}$ and $\mathscr{A}'$ in Figure \ref{fig11}, $\mathscr{A}_1$ and $\mathscr{A}_2$ differ only by the loop $l_b$. 
    
    Consider the local frame on $\mathcal{O}_1$ in Lemma \ref{lem4}, which can be extended continuously to $\mathcal{O}_2$ through the interior of $\sigma_b$. In Figure \ref{fig21}, we visualize the component $Z_{2^+3^+}$ of this local frame with $\mathscr{A}_1$ and $\mathscr{A}_2$. Let $\psi_1$ be defined as in Lemma \ref{lem4}, and let $\psi_2$ be the map from $[0,\infty)^8$ to this local frame via surfaces in $\mathcal{O}_2$. We orient the loops $l_b$ and $l_a$ in $\mathscr{A}_1$ and $\mathscr{A}_2$ according to the blue arrows, and observe from Figure \ref{fig21} that there exist unit vectors $\begin{pmatrix} b_x\\ b_y \end{pmatrix}$, $\begin{pmatrix} b'_x\\ b'_y \end{pmatrix}$ and $\begin{pmatrix} a_x\\ a_y \end{pmatrix}$ such that $Z_{2^+3^+}=b\begin{pmatrix} b_x\\ b_y \end{pmatrix}+a\begin{pmatrix} a_x\\ a_y \end{pmatrix}$ in $\psi_1$ and $Z_{2^+3^+}=b\begin{pmatrix} b'_x\\ b'_y \end{pmatrix}+a\begin{pmatrix} a_x\\ a_y \end{pmatrix}$ in $\psi_2$. 

    Analogous to the first example in Figure \ref{fig19}, from the directions of the blue arrows in Figure \ref{fig21}, we see that $\begin{vmatrix}b_x & a_x \\ b_y & a_y\end{vmatrix}<0$ and $\begin{vmatrix}b'_x & a_x \\ b'_y & a_y\end{vmatrix}>0$. Thus, by Lemma \ref{lem5},  $\overline{\mathcal{O}_1}$ and $\overline{\mathcal{O}_2}$ are on different sides of $\sigma_b$.
\end{proof}

\begin{figure}[h]
    \centering
    \includegraphics[width=0.75\linewidth]{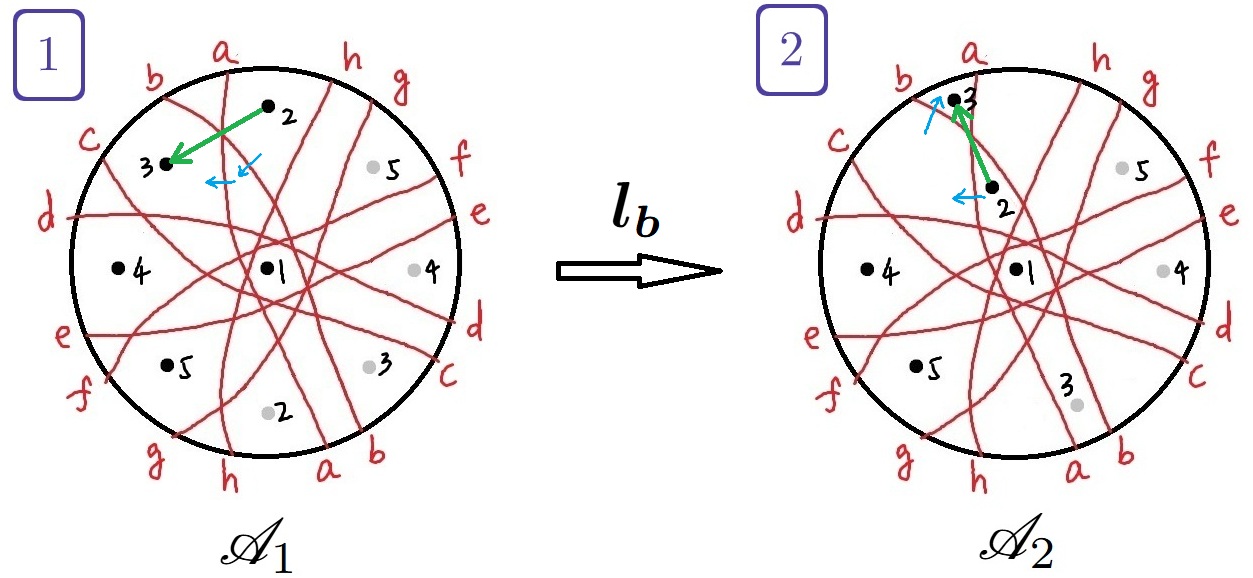}
    \caption{Visualizing $Z_{2^+3^+}$ with $\mathscr{A}_1$ and $\mathscr{A}_2$}
    \label{fig21}
\end{figure}

\begin{cor} \label{cor3}
Every Type 2 space is open in $\mathcal{C}_{10}(\delta_1, \delta_2, \delta_3, \delta_4, \delta_5)$.    
\end{cor}
\begin{proof}
   
    Without loss of generality, we prove the result for $\mathcal{O}_2$. Let $\psi_1$ and $\psi_2$ be the maps in the proof of Lemma \ref{lem7}. Referring to the proof of Lemma \ref{lem4}, we see that the key step is to show that $\psi_2$ is invertible. Let $M_1$ and $M_2$ be the matrices representing $\psi_1$ and $\psi_2$, respectively. By Lemma \ref{lem4}, $M_1$ is invertible. By Lemma \ref{lem7}, the determinants of $M_1$ and $M_2$ have opposite signs, so $M_2$ is invertible. The rest of proof is identical to that of Lemma \ref{lem4} in principle.
\end{proof}

\subsection*{Type 2 Spaces}
Without loss of generality, we can assume that $\mathcal{O}$ is $\mathcal{O}_2$. Observe from Figure \ref{fig22} that there are some symmetries among loops in $\mathscr{A}_2$. A reflection about the axis on the left interchanges $l_a$ and $l_b$, $l_c$ and $l_g$, $l_d$ and $l_h$. A reflection about the axis on the right interchanges $l_e$ and $l_f$, $l_c$ and $l_h$, $l_d$ and $l_g$. Thus, in summary, we can divide the $8$ loops into three symmetry classes: 1) $l_a$, $l_b$; 2) $l_e$ and $l_f$; 3) $l_c$, $l_d$, $l_g$ and $l_h$. Therefore, it suffices to discuss the cases when the face $\sigma$ is $\sigma_b$, $\sigma_e$ or $\sigma_h$. Since we have shown that $\overline{\mathcal{O}_1}$ and $\overline{\mathcal{O}_2}$ are on different sides of $\sigma_b$, we only need to discuss the cases $\sigma_e$ and $\sigma_h$.

\begin{figure}[h]
    \centering
    \includegraphics[width=0.7\linewidth]{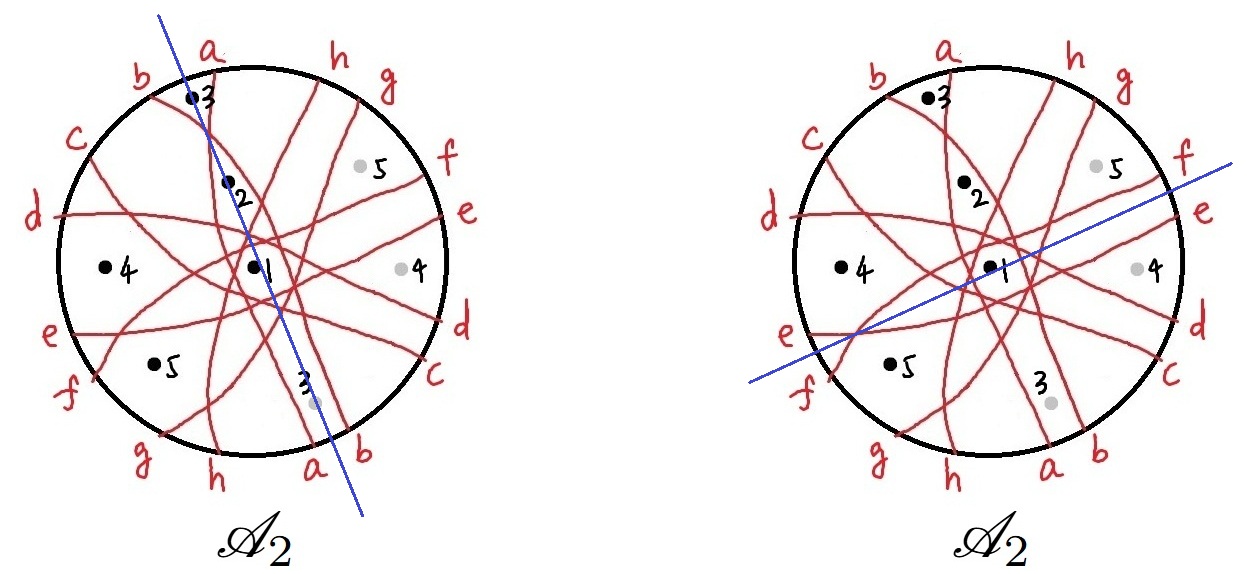}
    \caption{Two reflection symmetries of $\mathscr{A}_2$ about the axes in dark blue.}
    \label{fig22}
\end{figure}

\begin{lemma} \label{lem8}
    The closures $\overline{\mathcal{O}_2}$ and $\overline{\mathcal{O}_3}$ are on different sides of $\sigma_h$. 
\end{lemma}

\begin{proof}
    Observe from Figure \ref{fig23} that $\mathscr{A}_3$ is obtained from $\mathscr{A}_2$ by moving the labeled vertices $1^+$, $2^+$, $1^-$ and $2^-$ across $l_h$. Thus, $\mathscr{A}_2$ and $\mathscr{A}_3$ differ only by $l_h$, so $\sigma_h$ is a face of both $\overline{\mathcal{O}_2}$ and $\overline{\mathcal{O}_3}$.

    \begin{figure}[h]
    \centering
    \includegraphics[width=0.95\linewidth]{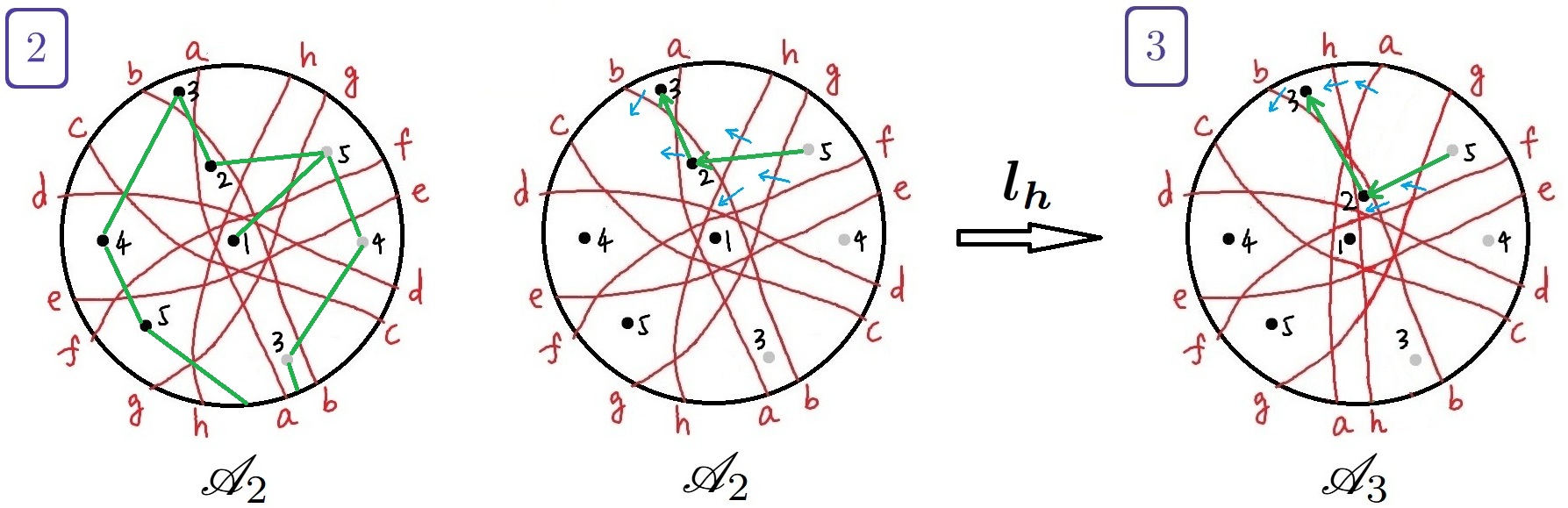}
    \caption{The unfolding demonstrated in the first picture gives rise to a local frame, whose components $Z_{5^-2^+}$ and $Z_{2^+3^+}$ are visualized with $\mathscr{A}_2$ and $\mathscr{A}_3$.}
    \label{fig23}
\end{figure}

    We consider the consistent way of cutting and unfolding surfaces in $\mathcal{O}_2$ demonstrated in the first picture of Figure \ref{fig23}. The green segments connecting $5^+$ to $2^-$ and $2^-$ to $3^-$ are not fully visible, but their antipodal images are drawn. This unfolding gives rise to a local frame $(Z_{5^-2^+}$, $Z_{2^+3^+}$, $Z_{3^+4^+}$, $Z_{4^+5^+})$ on $\mathcal{O}_2$, and can be extended to $\mathcal{O}_3$ through the interior of $\sigma_h$. Let $\psi_2$ and $\psi_3$ be the maps from $[0,\infty)^8$ to this local frame via surfaces in $\mathcal{O}_2$ and $\mathcal{O}_3$, respectively.

    In the second and the third pictures of Figure \ref{fig23}, we visualize the components $Z_{5^-2^+}$ and $Z_{2^+3^+}$ in this local frame with $\mathscr{A}_2$ and $\mathscr{A}_3$. By orienting the loops $l_h$, $l_a$, $l_g$ and $l_b$ in $\mathscr{A}_2$ and $\mathscr{A}_3$ according to the blue arrows, we observe that there exist unit vectors $\begin{pmatrix} h_x\\ h_y \end{pmatrix}$, $\begin{pmatrix} h'_x\\ h'_y \end{pmatrix}$, $\begin{pmatrix} a_x\\ a_y \end{pmatrix}$, $\begin{pmatrix} g_x\\ g_y \end{pmatrix}$ and $\begin{pmatrix} b_x\\ b_y \end{pmatrix}$ such that $Z_{5^-2^+}=h\begin{pmatrix} h_x\\ h_y \end{pmatrix}+g\begin{pmatrix} g_x\\ g_y \end{pmatrix}+b\begin{pmatrix} b_x\\ b_y \end{pmatrix}$ and $Z_{2^+3^+}=a\begin{pmatrix} a_x\\ a_y \end{pmatrix}-b\begin{pmatrix} b_x\\ b_y \end{pmatrix}$ in $\psi_2$. In addition, $Z_{5^-2^+}=g\begin{pmatrix} g_x\\ g_y \end{pmatrix}+b\begin{pmatrix} b_x\\ b_y \end{pmatrix}$ and $Z_{2^+3^+}=h\begin{pmatrix} h'_x\\ h'_y \end{pmatrix}+a\begin{pmatrix} a_x\\ a_y \end{pmatrix}-b\begin{pmatrix} b_x\\ b_y \end{pmatrix}$ in $\psi_3$. Thus, we can apply Lemma \ref{lem6} by taking $T=5^-$, $U=V=2^+$, $W=3^+$, $i=h$, $j=a$, $k=g$, and $m=b$. 

     We can compute the products of the determinants in Lemma \ref{lem6} by the directions of the blue arrows in Figure \ref{fig23}. But we can also compare Figure \ref{fig23} with Figure \ref{fig20} directly to see that they are in the same situation. Therefore, $\overline{\mathcal{O}_2}$ and $\overline{\mathcal{O}_3}$ are on different sides of $\sigma_h$ by Lemma \ref{lem6}.
    \end{proof}

\begin{cor} \label{cor4}
The interior of a Type 3 space is open in $\mathcal{C}_{10}(\delta_1, \delta_2, \delta_3, \delta_4, \delta_5)$.      
\end{cor}
\begin{proof}
    Without loss of generality, we prove the result for $\mathcal{O}_3$. Consider the maps $\psi_2$ and $\psi_3$ in the proof of Lemma \ref{lem8}. Let $M_2$ and $M_3$ be the matrices that represent $\psi_2$ and $\psi_3$, respectively. By Corollary \ref{cor3}, $\mathcal{O}_2$ is open in $\mathcal{C}_{10}(\delta_1, \delta_2, \delta_3, \delta_4, \delta_5)$, so $M_2$ is invertible. By Lemma \ref{lem8}, the determinants of $M_2$ and $M_3$ have opposite signs, so $M_3$ is invertible. The rest of the proof is identical to that of Lemma \ref{lem4} in principle.    
\end{proof}

\begin{lemma} \label{lem9}
    The closures $\overline{\mathcal{O}_2}$ and $\overline{\mathcal{O}_4}$ are on different sides of $\sigma_e$.
\end{lemma}

\begin{proof}
    Observe from Figure \ref{fig24} that $\mathscr{A}_4$ is obtained from $\mathscr{A}_2$ by moving the labeled vertices $4^+$, $4^-$, $5^+$, and $5^-$  across the loop $l_e$. Therefore, $\mathscr{A}_2$ and $\mathscr{A}_4$ differ only by $l_e$, so $\sigma_e$ is a face of both $\overline{\mathcal{O}_2}$ and $\overline{\mathcal{O}_4}$. 
    
    \begin{figure}[h]
    \centering
    \includegraphics[width=0.75\linewidth]{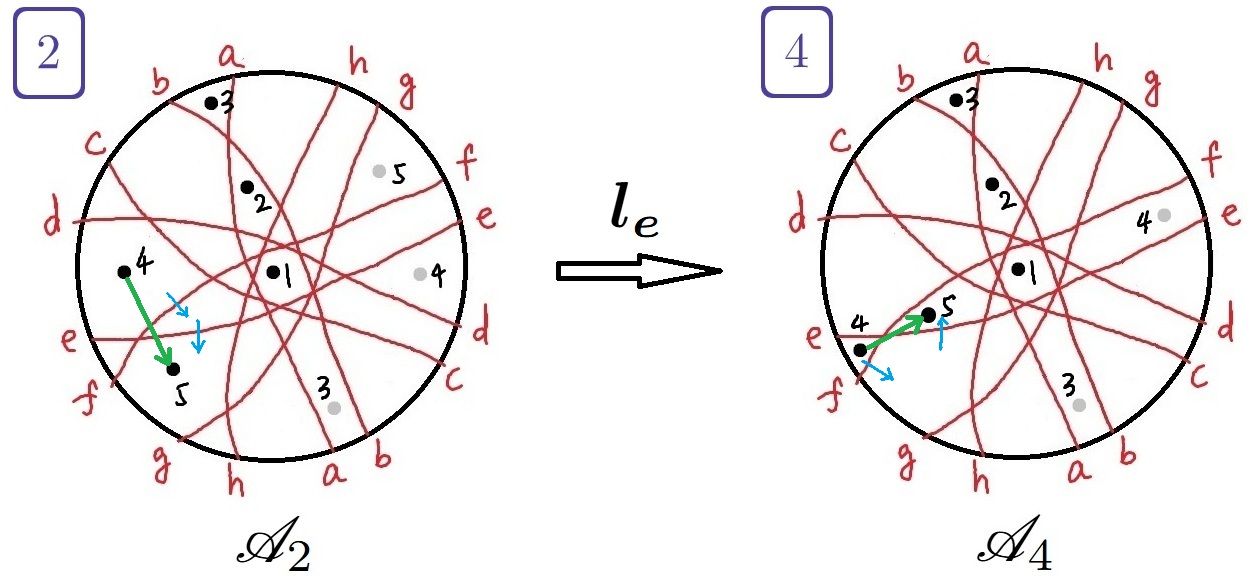}
    \caption{Visualizing the component $Z_{4^+5^+}$ with $\mathscr{A}_2$ and $\mathscr{A}_4$.}
    \label{fig24}
    \end{figure}
    
    To see that $\overline{\mathcal{O}_4}$ are on different sides of $\sigma_e$, we use the same local frame on $\mathcal{O}_2$ in the proof of Lemma \ref{lem8}, and extend it to $\mathcal{O}_4$ through the interior of $\sigma_e$. Let $\psi_2$ and $\psi_4$ be the maps from $[0,\infty)^8$ to this local frame via surfaces in $\mathcal{O}_2$ and $\mathcal{O}_4$, respectively. By visualizing the component $Z_{4^+5^+}$ of this local frame with $\mathscr{A}_2$ and $\mathscr{A}_4$ in Figure \ref{fig24}, we can find unit vectors $\begin{pmatrix} e_x \\ e_y \end{pmatrix}$, $\begin{pmatrix} e'_x \\ e'_y \end{pmatrix}$ and $\begin{pmatrix} f_x \\ f_y \end{pmatrix}$ such that $Z_{4^+5^+}=e\begin{pmatrix} e_x\\ e_y \end{pmatrix}+f\begin{pmatrix} f_x\\ f_y \end{pmatrix}$ in $\psi_2$ and $Z_{4^+5^+}=e\begin{pmatrix} e'_x\\ e'_y \end{pmatrix}+f\begin{pmatrix} f_x\\ f_y \end{pmatrix}$ in $\psi_4$. 

    Then we orient $l_e$ and $l_f$ based on the formulas of $Z_{4^+5^+}$ above, and mark their orientations by blue arrows in Figure \ref{fig24}. We can see that $\begin{vmatrix}e_x & f_x \\ e_y & f_y\end{vmatrix}>0$ and $\begin{vmatrix}e'_x & f_x \\ e'_y & f_y\end{vmatrix}<0$, so the result follows from Lemma \ref{lem5} by taking $T=4^+$, $U=5^+$, $i=e$ and $j=f$.
\end{proof}

By an argument similar to Corollary \ref{cor3} and \ref{cor4}, we obtain the following result:

\begin{cor} \label{cor5} 
    The interior of a Type 4 space is open in $\mathcal{C}_{10}(\delta_1, \delta_2, \delta_3, \delta_4, \delta_5)$.  
\end{cor}

\subsection*{Type 3 Spaces}

Without loss of generality, we can assume that $\mathcal{O}$ is the space $\mathcal{O}_3$ arising from $\mathscr{A}_3$. Symmetries among loops in $\mathscr{A}_3$ are less obvious. In Figure \ref{fig25}, we sketch $\mathscr{A}_3$ in the first picture. Then we push all the labeled vertices in $\mathscr{A}_3$ to the equator of the sphere as the second picture shows. 

    \begin{figure}[h]
    \centering    \includegraphics[width=0.85\linewidth]{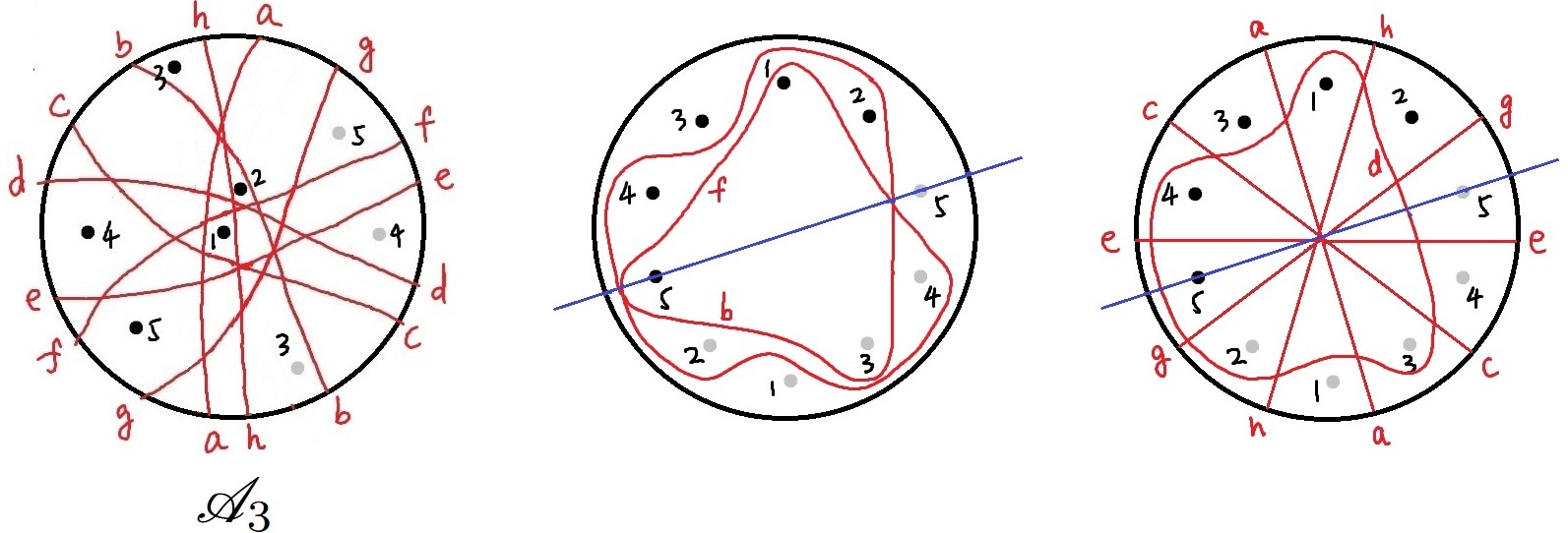}
    \caption{When pushing all the labeled vertices to the equator, loops in $\mathscr{A}_3$ are moved or distorted so that their homotopy classes remain unchanged.}
    \label{fig25}
    \end{figure}

When pushing the vertices, we also move or distort some loops in $\mathscr{A}_3$ so that their homotopy classes remain the same. After that, we draw the loops $l_b$ and $l_f$ in the second picture (in which both loops are distorted), and the remaining loops in the third picture (in which $l_d$ is distorted). To verify both pictures, it would be helpful to consider which vertices are on the same side of each loop in $\mathscr{A}_3$.

The second and the third pictures in Figure \ref{fig25} has a reflection symmetry about the dark blue axis. We can thus divide the loops in $\mathscr{A}_3$ into five symmetry classes: 1) $l_a$; 2) $l_b$ and $l_f$; 3) $l_c$ and $l_h$; 4) $l_d$; 5) $l_e$ and $l_g$. Since $\overline{\mathcal{O}_2}$ is on the other side of $\sigma_h$ with respect to $\overline{\mathcal{O}_3}$ by Lemma \ref{lem9}, it suffices to discuss the cases when $\sigma$ is $\sigma_a$, $\sigma_b$, $\sigma_d$ or $\sigma_e$

\begin{lemma} \label{lem10} 
    There is a Type 3 space whose closure is on the other side of $\sigma_a$ with respect to $\overline{\mathcal{O}_3}$.
\end{lemma}
\begin{proof}
    Consider the loop arrangement $\mathscr{A}'$ in the third picture of Figure \ref{fig26}. It is obtained from $\mathscr{A}_3$ by moving the loop $l_a$ and perturbing the loop $l_g$ with the labeled vertices $5^+$ and $5^-$. Note that this does not change the homotopy class of $l_g$, so $\mathscr{A}'$ and $\mathscr{A}_3$ differ only by $l_a$. Let $\mathcal{O}'$ be the space arising from $\mathscr{A}'$ and $\overline{\mathcal{O}'}$ be its closure. Then $\sigma_a$ is a common face of $\overline{\mathcal{O}_3}$ and $\overline{\mathcal{O}'}$.

    \begin{figure}[h]
    \centering
    \includegraphics[width=0.95\linewidth]{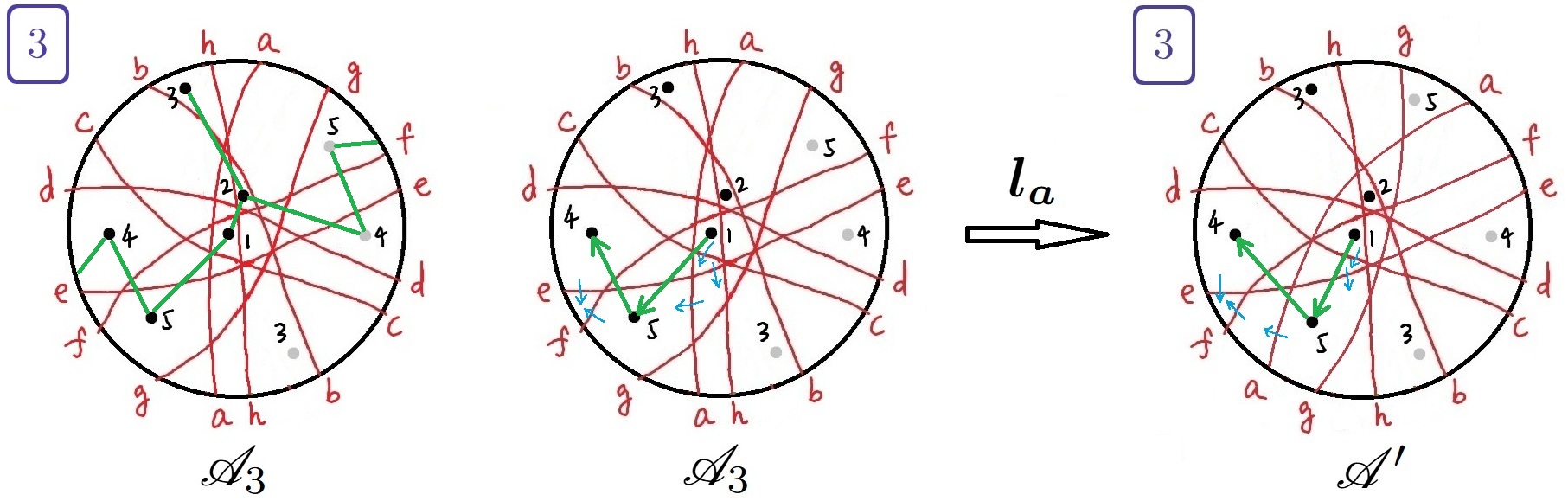}
    \caption{The unfolding demonstrated in the first picture gives rise to a local frame, whose components $Z_{1^+5^+}$ and $Z_{5^+4^+}$ are visualized with $\mathscr{A}_3$ and $\mathscr{A}'$.}
    \label{fig26}
    \end{figure}

    Consider the way of cutting and unfolding surfaces in $\mathcal{O}_3$ demonstrated in the first picture of Figure \ref{fig26}. The green segments connecting $4^+$ to $2^-$, $2^-$ to $1^-$, and $1^-$ to $5^-$ are not fully visible, but we can see their antipodal images. This unfolding gives rise to a local frame $(Z_{4^-2^+}$, $Z_{2^+1^+}$, $Z_{1^+5^+}$, $Z_{5^+4^+})$ on $\mathcal{O}_3$, and is extended to $\mathcal{O}'$ through the interior of $\sigma_a$. In Figure \ref{fig26}, we visualize the components $Z_{1^+5^+}$ and $Z_{5^+4^+}$ of this local frame with $\mathscr{A}_3$ and $\mathscr{A}'$.
    
    To show that $\overline{\mathcal{O}_3}$ and $\overline{\mathcal{O}'}$ are on different sides of $\sigma_3$, we orient the loops $l_a$, $l_c$, $l_f$ and $l_e$ in $\mathscr{A}_3$ and $\mathscr{A}'$ according to the blue arrows in Figure \ref{fig26}. Let $\psi_3$ and $\psi'$ be the maps from $[0,\infty)^8$ to the local frame above via surfaces in $\mathcal{O}_3$ and $\mathcal{O}'$, respectively. Then we can see that there exist unit vectors $\begin{pmatrix} a_x\\ a_y \end{pmatrix}$, $\begin{pmatrix} a'_x\\ a'_y \end{pmatrix}$, $\begin{pmatrix} c_x\\ c_y \end{pmatrix}$, $\begin{pmatrix} f_x\\ f_y \end{pmatrix}$ and $\begin{pmatrix} e_x\\ e_y \end{pmatrix}$ such that $Z_{1^+5^+}=a\begin{pmatrix} a_x\\ a_y \end{pmatrix}+c\begin{pmatrix} c_x\\ c_y \end{pmatrix}+e\begin{pmatrix} e_x\\ e_y \end{pmatrix}$ and $Z_{5^+4^+}=f\begin{pmatrix} f_x\\ f_y \end{pmatrix}-e\begin{pmatrix} e_x\\ e_y \end{pmatrix}$ in $\psi_3$. In addition, $Z_{1^+5^+}=c\begin{pmatrix} c_x\\ c_y \end{pmatrix}+e\begin{pmatrix} e_x\\ e_y \end{pmatrix}$ and  $Z_{5^+4^+}=a\begin{pmatrix} a'_x\\ a'_y \end{pmatrix}+f\begin{pmatrix} f_x\\ f_y \end{pmatrix}-e\begin{pmatrix} e_x\\ e_y \end{pmatrix}$ in $\psi'$.

    Therefore, we can apply Lemma \ref{lem6} by taking $T=1^+$, $U=V=5^+$, $W=4^+$, $l_i=l_a$, $l_j=l_c$, $l_k=l_f$ and $l_m=l_e$. This is exactly the same situation as Figure \ref{fig20}, so $\overline{\mathcal{O}_3}$ and $\overline{\mathcal{O}'}$ are on different sides of $\sigma_a$.

    Finally, we show that $\mathcal{O}'$ is a Type 3 space by showing that $\mathscr{A}'$ is combinatorically equivalent to $\mathscr{A}_3$. We move the labeled vertices in $\mathscr{A}'$ to the equator of the sphere as Figure \ref{fig27} shows, and move or distort loops if necessary so that their homotopy classes remain unchanged. One can verify the second and the third pictures of Figure \ref{fig27} by comparing the vertices on the same side of each loop in $\mathscr{A}'$. Then we compare these two pictures with the last two pictures in Figure \ref{fig25}, we see that they differ by a rotation by $\pi$ when we ignore all the labels. 
\end{proof}

\begin{figure}[h]
    \centering
    \includegraphics[width=0.9\linewidth]{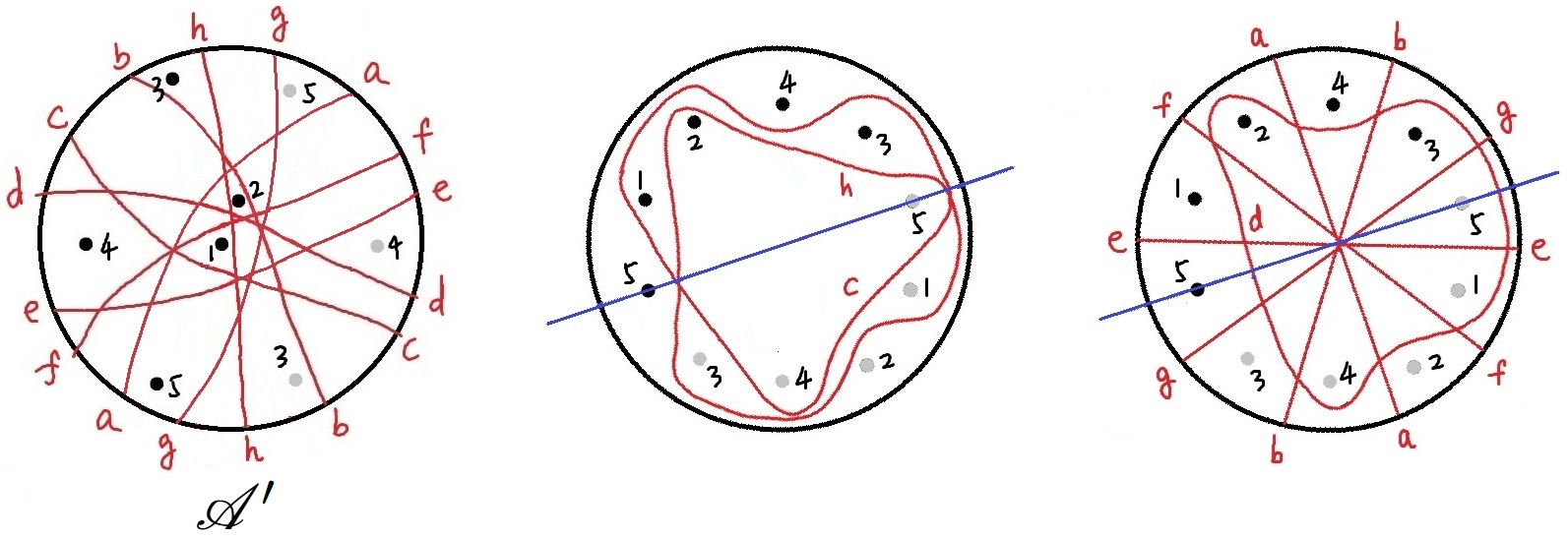}
    \caption{The loop arrangement $\mathscr{A}'$ is combinatorically equivalent to $\mathscr{A}_3$.}
    \label{fig27}
\end{figure}

\begin{lemma} \label{lem11}
    There is a Type 1 space whose closure is on the other side of $\sigma_b$ with respect to $\overline{\mathcal{O}_3}$.
\end{lemma}
\begin{proof}
    Consider the loop arrangement $\mathscr{A}'$ obtained from $\mathscr{A}_3$ by moving the loops $l_b$, $l_g$ and the labeled vertices $2^+$ and $2^-$ to their positions in the second picture of Figure \ref{fig28}. Note that this only changes the homotopy class of $l_b$. Let $\mathcal{O}'$ be the space arising from $\mathscr{A}'$ and $\overline{\mathcal{O}'}$ be its closure. Then $\sigma_b$ is a common face of $\overline{\mathcal{O}_3}$ and $\overline{\mathcal{O}'}$.

    \begin{figure}[h]
    \centering
    \includegraphics[width=0.95\linewidth]{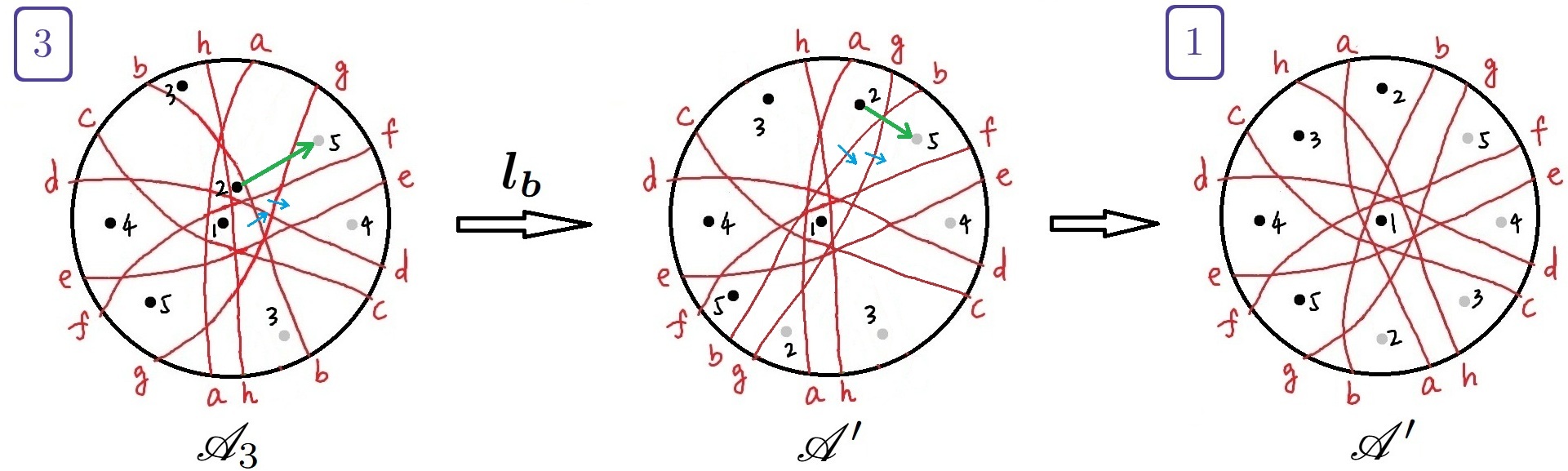}
    \caption{Visualizing $Z_{2^+5^-}$ on $\mathscr{A}_3$ and $\mathscr{A}'$, which is combinatorically equivalent to $\mathscr{A}_3$.}
    \label{fig28}
\end{figure}


    Consider the local frame $(Z_{4^-2^+}$, $Z_{2^+1^+}$, $Z_{1^+5^+}$, $Z_{5^+4^+})$ on $\mathcal{O}_3$ in Lemma \ref{lem10}. If we replace the component $Z_{5^+4^+}$ by $Z_{5^-4^-}$ and then by $-(Z_{5^-4^-}+Z_{4^-2^+})=Z_{2^+5^-}$, we get another frame on $\mathcal{O}_3$. We extend this local frame to $\mathcal{O}'$ through the interior of $\sigma_b$, and visualize the component $Z_{2^+5^-}$ with $\mathscr{A}_3$ and $\mathscr{A}'$ in Figure \ref{fig28}.

    To see that $\overline{\mathcal{O}_3}$ and $\overline{\mathcal{O}'}$ are on different sides of $\sigma_b$, we orient the loops $l_b$ and $l_g$ in $\mathscr{A}_3$ and $\mathscr{A}'$ according to the blue arrows in Figure \ref{fig28}. Let $\psi_3$ and $\psi'$ be the maps from $[0,\infty)^8$ to the above local frame via surfaces in $\mathscr{A}_3$ and $\mathscr{A}'$, respectively. Then we can observe from Figure \ref{fig28} that there exist unit vectors $\begin{pmatrix} b_x\\ b_y \end{pmatrix}$, $\begin{pmatrix} b'_x\\ b'_y \end{pmatrix}$ and $\begin{pmatrix} g_x\\ g_y \end{pmatrix}$ such that $Z_{2^+5^-}=b\begin{pmatrix} b_x\\ b_y \end{pmatrix}+g\begin{pmatrix} g_x\\ g_y \end{pmatrix}$ in $\psi_3$ and $Z_{2^+5^-}=b\begin{pmatrix} b'_x\\ b'_y \end{pmatrix}+g\begin{pmatrix} g_x\\ g_y \end{pmatrix}$ in $\psi'$. Thus, we can apply Lemma \ref{lem5} by taking $U=2^+$, $V=5^-$, $i=b$ and $j=g$.

    According to the directions of the blue arrows in Figure \ref{fig28}, we can see that $\begin{vmatrix}b_x & g_x \\ b_y & g_y\end{vmatrix}<0$ and $\begin{vmatrix}b'_x & g_x \\ b'_y & g_y\end{vmatrix}>0$, so $\overline{\mathcal{O}_3}$ and $\overline{\mathcal{O}'}$ are on different sides of $\sigma_b$. 

    Finally, we show that $\mathcal{O}'$ is a Type 1 space. We just perturb some loops and labeled vertices in $\mathscr{A}'$ to their positions in the third picture of Figure \ref{fig28}. This does not change the homotopy class of any loop, so $\mathscr{A}'$ is combinatorically equivalent to $\mathscr{A}_1$. 
\end{proof}

\begin{lemma} \label{lem12}
    There is a Type $4$ space whose closure is on the other side of $\sigma_d$ with respect to $\overline{\mathcal{O}_3}$.
\end{lemma}
\begin{proof}
    Consider the loop arrangement $\mathscr{A}'$ obtained by moving the loop $l_d$ to its position in the second picture of Figure \ref{fig29}. Let $\mathcal{O}'$ be the space arising from $\mathscr{A}'$ and $\overline{\mathcal{O}'}$ be its closure. Then $\sigma_d$ is a common face of $\overline{\mathcal{O}_3}$ and $\overline{\mathcal{O}'}$. 

    \begin{figure}[h]
    \centering
    \includegraphics[width=0.95\linewidth]{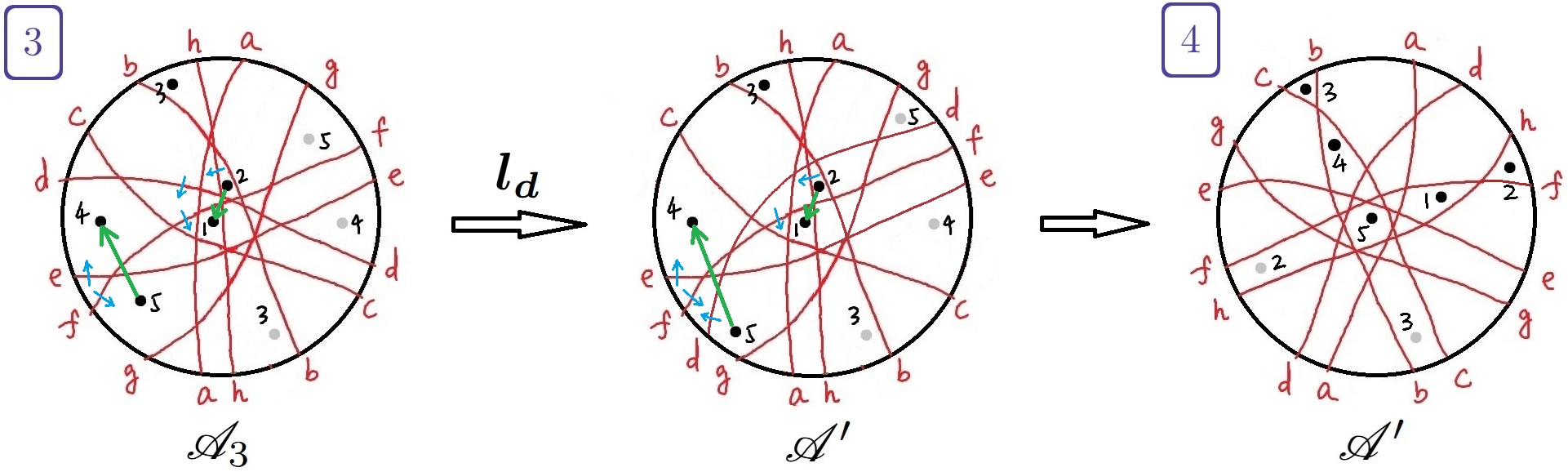}
    \caption{Visualizing $Z_{2^+1^+}$ and $Z_{5^+4^+}$ with $\mathscr{A}_3$ and $\mathscr{A}'$ (combinatorically equivalent to $\mathscr{A}_4$).}
    \label{fig29}
\end{figure}

    Consider the local frame in Lemma \ref{lem10} on $\mathcal{O}_3$ (and extended to $\mathcal{O}'$ through the interior of $\sigma_d$). The components $Z_{2^+1^+}$ and $Z_{5^+4^+}$ of this local frame are visualized with $\mathscr{A}_3$ and $\mathscr{A}'$ in Figure \ref{fig29}. Then we orient the loops $l_d$, $l_h$, $l_e$ and $l_f$ according to the blue arrows. Let $\psi_3$ and $\psi'$ be the maps from $[0,\infty)^8$ to the above local frame via surfaces in $\mathscr{A}_3$ and $\mathscr{A}'$, respectively. From Figure \ref{fig29}, we can see that there exist unit vectors $\begin{pmatrix} d_x\\ d_y \end{pmatrix}$, $\begin{pmatrix} d'_x\\ d'_y \end{pmatrix}$, $\begin{pmatrix} h_x\\ h_y \end{pmatrix}$, $\begin{pmatrix} e_x\\ e_y \end{pmatrix}$ and $\begin{pmatrix} f_x\\ f_y \end{pmatrix}$ such that $Z_{2^+1^+}=d\begin{pmatrix} d_x\\ d_y \end{pmatrix}+h\begin{pmatrix} h_x\\ h_y \end{pmatrix}+f\begin{pmatrix} f_x\\ f_y \end{pmatrix}$ and $Z_{5^+4^+}=e\begin{pmatrix} e_x\\ e_y \end{pmatrix}-f\begin{pmatrix} f_x\\ f_y \end{pmatrix}$ in $\psi_3$. In addition, $Z_{2^+1^+}=h\begin{pmatrix} h_x\\ h_y \end{pmatrix}+f\begin{pmatrix} f_x\\ f_y \end{pmatrix}$ and $Z_{5^+4^+}=d\begin{pmatrix} d'_x\\ d'_y \end{pmatrix}+e\begin{pmatrix} e_x\\ e_y \end{pmatrix}-f\begin{pmatrix} f_x\\ f_y \end{pmatrix}$  in $\psi'$. Therefore, we can apply Lemma \ref{lem6} by taking $T=2^+$, $U=1^+$, $V=5^+$, $W=4^+$, $i=d$, $j=h$, $k=e$ and $m=f$.

    According to the directions of the blue arrows in Figure \ref{fig29}, we conclude that $\begin{vmatrix}d_x & h_x \\ d_y & h_y\end{vmatrix}\cdot \begin{vmatrix}e_x & f_x \\ e_y & f_y\end{vmatrix}>0$ and $\begin{vmatrix}d'_x & e_x \\ d'_y & e_y\end{vmatrix}\cdot\begin{vmatrix}h_x & f_x \\ h_y & f_y\end{vmatrix}<0$. Thus $\overline{\mathcal{O}_3}$ and $\overline{\mathcal{O}'}$ are on different sides of $\sigma_d$.     
    
    Finally, $\mathscr{A}'$ is combinatorically equivalent to $\mathscr{A}_4$. We may rotate the sphere so that the vertex $5^+$ is at the center of our view, and then perturb the loops in $\mathscr{A}'$ to their positions in the third picture of Figure \ref{fig29}. In fact, this does not change the homotopy class of each loop. Although we have not found a way to demonstrate this easily without a 3D model, it is still helpful to compare the vertices on the same side of each loop in the second and the third pictures. Finally, the third picture is equivalent to $\mathscr{A}_4$ up to a rotation if we ignore all the labels. 
\end{proof}

\begin{lemma} \label{lem13}
    There is a Type 3 space whose closure is on the other side of $\sigma_e$ with respect to $\overline{\mathcal{O}_3}$.
\end{lemma}
\begin{proof}
    Consider the loop arrangement $\mathscr{A}'$ obtained by moving the loops in $\mathscr{A}_3$ to their positions in the second picture of Figure \ref{fig30}. Note that this does not change the homotopy classes of any loop except for $l_e$. Let $\mathcal{O}'$ be the space arising from $\mathscr{A}'$ and $\overline{\mathcal{O}'}$ be its closure. Then $\sigma_e$ is a common face of $\overline{\mathcal{O}_3}$ and $\overline{\mathcal{O}'}$.     

    \begin{figure}[h]
    \centering
    \includegraphics[width=0.95\linewidth]{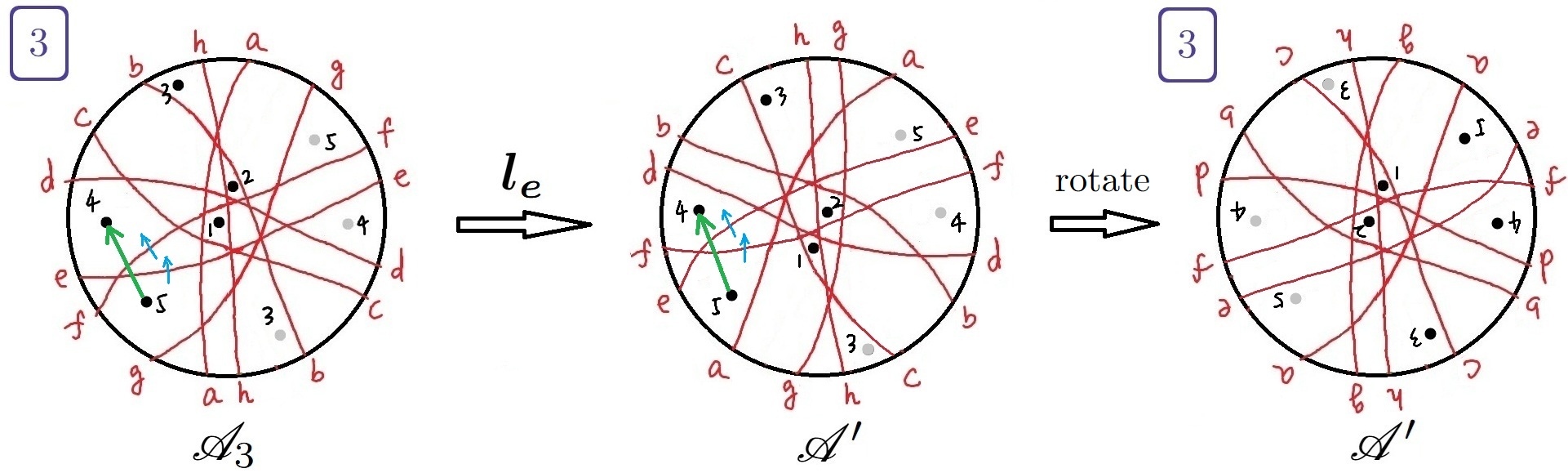}
    \caption{Visualizing $Z_{5^+4^+}$ with $\mathscr{A}_3$ and $\mathscr{A}'$, which is combinatorically equivalent to $\mathscr{A}_3$.}
    \label{fig30}
    \end{figure}

    Consider the local frame in Lemma \ref{lem10} on $\mathcal{O}_3$, and extend it to $\mathcal{O}'$ through the interior of $\sigma_e$. In Figure \ref{fig30}, we visualize the component $Z_{5^+4^+}$ of this local frame with $\mathscr{A}_3$ and $\mathscr{A}'$, and orient the loops $l_e$ and $l_g$ according to the blue arrows. Let $\psi_3$ and $\psi'$ be the maps from $[0,\infty)^8$ to this local frame via surfaces in $\mathscr{A}_3$ and $\mathscr{A}'$, respectively. Then there exist unit vectors $\begin{pmatrix} e_x\\ e_y \end{pmatrix}$, $\begin{pmatrix} e'_x\\ e'_y \end{pmatrix}$, $\begin{pmatrix} f_x\\ f_y \end{pmatrix}$ such that  $Z_{5^+4^+}=e\begin{pmatrix} e_x\\ e_y \end{pmatrix}+f\begin{pmatrix} f_x\\ f_y \end{pmatrix}$ in $\psi_3$ and $Z_{5^+4^+}=e\begin{pmatrix} e'_x\\ e'_y \end{pmatrix}+f\begin{pmatrix} f_x\\ f_y \end{pmatrix}$ in $\psi'$. 
    
    Thus, we can apply Lemma \ref{lem5} by taking $T=4^+$, $U=5^+$, $i=e$ and $j=f$. Based on the directions of the blue arrows in Figure \ref{fig30}, we can see that $\begin{vmatrix}e_x & f_x \\ e_y & f_y\end{vmatrix}>0$ and $\begin{vmatrix}e'_x & f_x \\ e'_y & f_y\end{vmatrix}<0$. Thus, $\overline{\mathcal{O}_3}$ and $\overline{\mathcal{O}'}$ are on different sides of $\sigma_e$.

    Finally, $\mathscr{A}'$ is combinatorically equivalent to $\mathscr{A}_3$. To see this, we rotate the second picture of $\mathscr{A}'$ by $\pi$ to obtain the third picture in Figure \ref{fig30}, which is equivalent to $\mathscr{A}_3$ if we ignore the labels. 
\end{proof}

We have thus finished the discussion with Type $3$ spaces.

\subsection*{Type 4 Spaces}

In the picture of $\mathscr{A}_4$ in Figure \ref{fig17}, we observe that there are essentially two symmetry classes of loops: 1) $l_a$, $l_b$, $l_e$ and $l_f$; 2) $l_c$, $l_d$, $l_g$ and $l_h$. In Lemma \ref{lem9}, we have shown that there exists a Type 2 space whose closure is on the other side of $\sigma_e$ with respect to  $\overline{\mathcal{O}_4}$. In Lemma \ref{lem12}, we have shown that there exists a Type 3 space whose closure is on the other side of $\sigma_d$ with respect to $\overline{\mathcal{O}_4}$. This completes the discussion with Type 4 spaces.

Let us summarize the results from the lemmas and the corollaries in this section so far:

\begin{lemma} \label{lem14}
    Let ${\mathcal{O}}$ be a space of Type 1, 2, 3 or 4 and $\overline{\mathcal{O}}$ be its closure. Then ${\mathcal{O}}$ is open in $\mathcal{C}_{10}(\delta_1, \delta_2, \delta_3, \delta_4, \delta_5)$. In addition, for any face $\sigma$ of $\overline{\mathcal{O}}$, there exists a space of Type 1, 2, 3 or 4 whose closure is on the other side of $\sigma$ with respect to $\overline{\mathcal{O}}$.
\end{lemma}

Analogous to Section 4, we consider the moduli space $\Delta$ of $\mathcal{O}$, the space of surfaces in $\mathcal{O}$ with area $1$. Let $\overline{\Delta}$ be the completion (with respect to the real hyperbolic metric) of $\Delta$. When $\mathcal{O}$ has Type 1, 2, 3 or 4, a similar argument to Lemma \ref{lem3} shows that $\overline{\Delta}$ is a real hyperbolic ideal $7$-simplex.

Consider the moduli space $\mathcal{M}_{10}(\delta_1, \delta_2, \delta_3, \delta_4, \delta_5)$ and its metric completion $\overline{\mathcal{M}_{10}}(\delta_1, \delta_2, \delta_3, \delta_4, \delta_5)$. Then Lemma \ref{lem14} implies the following result analogous to Corollary \ref{cor1}:

\begin{cor} \label{cor6}
    If $\mathcal{O}$ has Type 1, 2, 3 or 4, then every element in $\Delta$ or in the interior of a face of $\overline{\Delta}$ has an open neighborhood in $\mathcal{M}_{10}(\delta_1, \delta_2, \delta_3, \delta_4, \delta_5)$.
\end{cor}

Finally, we are ready to prove the following result, which is an analog of Theorem \ref{thm4}:

\begin{theorem} \label{thm5}
    Every polyhedral surface in $\overline{\mathcal{M}_{10}}(\delta_1, \delta_2, \delta_3, \delta_4, \delta_5)$ can be decomposed into at most $2\binom{8}{2}=56$ parallelograms, and the decomposition is invariant under the antipodal map.
\end{theorem}

\begin{proof}
    The idea of proof is essentially the same as that of Theorem \ref{thm4}. We construct infinitely many copies of the moduli spaces of Type 1, 2, 3 or 4. Then we glue their metric completions (which are $7$-simplices) to form a space $\overline{X}$ such that every face belongs to exactly two $7$-simplices, one on each side. By Lemma \ref{lem14}, such $\overline{X}$ exists. Then we define the map $\iota$ analogously. Away from the codimension-two boundaries of $\overline{X}$, the openness of $\iota$ is guaranteed by Corollary \ref{cor6}. Finally, there are $8$ loops in every loop arrangement, each two of which intersect exactly twice as great circles, hence the number $2\binom{8}{2}$.
\end{proof}

\section{Discussions and Future Work}

Our proof of Theorem \ref{thm5} is based on Corollary \ref{cor6}. In fact, this method of proof has the potential to be applied in higher dimensions as long as we can generalize Lemma \ref{lem14} (and thus Corollary \ref{cor6}) to higher dimensions. To achieve this, we will need to construct different types of loop arrangements. In addition, for each loop arrangement $\mathscr{A}$ and any loop $l_i$ in $\mathscr{A}$, we need to move $l_i$ to an appropriate position to get an adjacent loop arrangement $\mathscr{A}'$ such that the closures of the spaces arising from $\mathscr{A}$ and $\mathscr{A}'$ are on different sides of the face $\sigma_i$. The difficulty is that we cannot always find a position to move $l_i$ by induction, even though we may remove another two loops $l_j$ and $l_k$ from $\mathscr{A}$ to obtain a loop arrangement of a lower dimensional space. This is because when we move $l_i$ by the induction hypothesis, we may move some labeled vertices at the same time, making it impossible to add $l_j$ and $l_k$ back so that they remain great circles in their original homotopy classes. Without induction, it would be tedious to enumerate and check all types of loop arrangements as we can expect from the work in Section 5.

Nevertheless, we believe that Theorem \ref{thm5} can be generalized to higher dimensions. We summarize it as the conjecture below to work on in the future:

\begin{conjecture}
    Let $\delta_1$, $\delta_2$, $\dots$, $\delta_N$ be $N$ positive numbers that sum up to $2\pi$. Then every surface in $\overline{\mathcal{M}_{2N}}(\delta_1, \delta_2, \dots, \delta_N)$ can be decomposed into at most $2\binom{2N-2}{2}$ parallelograms, and the decomposition is invariant under the antipodal map.
\end{conjecture}

\section*{Acknowledgments}

We are very grateful to Professor Richard Schwartz and Professor Peter Doyle for their very helpful comments and support.

\end{document}